\documentclass[11pt]{article}
\usepackage{latexsym,amsfonts,amssymb,amsmath,amsthm}
\usepackage{graphicx}
\usepackage{cite}
\usepackage[usenames,dvipsnames]{color}
\usepackage{ulem}
\usepackage{amsmath,amscd}

\begin{document}
\setlength{\baselineskip}{16pt}

\parindent 0.5cm
\evensidemargin 0cm \oddsidemargin 0cm \topmargin 0cm \textheight
22cm \textwidth 16cm \footskip 2cm \headsep 0cm

\newtheorem{theorem}{Theorem}[section]
\newtheorem{lemma}[theorem]{Lemma}
\newtheorem{proposition}[theorem]{Proposition}
\newtheorem{definition}{Definition}[section]
\newtheorem{example}{Example}[section]
\newtheorem{corollary}[theorem]{Corollary}

\newtheorem{remark}{Remark}[section]
\newtheorem{property}[theorem]{Property}
\numberwithin{equation}{section}
\newtheorem{mainthm}{Theorem}
\newtheorem{mainlem}{Lemma}

\numberwithin{equation}{section}

\def\p{\partial}
\def\I{\textit}
\def\R{\mathbb R}
\def\C{\mathbb C}
\def\u{\underline}
\def\l{\lambda}
\def\a{\alpha}
\def\O{\Omega}
\def\e{\epsilon}
\def\ls{\lambda^*}
\def\D{\displaystyle}
\def\wyx{ \frac{w(y,t)}{w(x,t)}}
\def\imp{\Rightarrow}
\def\tE{\tilde E}
\def\tX{\tilde X}
\def\tH{\tilde H}
\def\tu{\tilde u}
\def\d{\mathcal D}
\def\aa{\mathcal A}
\def\DH{\mathcal D(\tH)}
\def\bE{\bar E}
\def\bH{\bar H}
\def\M{\mathcal M}
\renewcommand{\labelenumi}{(\arabic{enumi})}

\def\disp{\displaystyle}
\def\undertex#1{$\underline{\hbox{#1}}$}
\def\card{\mathop{\hbox{card}}}
\def\sgn{\mathop{\hbox{sgn}}}
\def\exp{\mathop{\hbox{exp}}}
\def\OFP{(\Omega,{\cal F},\PP)}
\newcommand\JM{Mierczy\'nski}
\newcommand\RR{\ensuremath{\mathbb{R}}}
\newcommand\CC{\ensuremath{\mathbb{C}}}
\newcommand\QQ{\ensuremath{\mathbb{Q}}}
\newcommand\ZZ{\ensuremath{\mathbb{Z}}}
\newcommand\NN{\ensuremath{\mathbb{N}}}
\newcommand\PP{\ensuremath{\mathbb{P}}}
\newcommand\abs[1]{\ensuremath{\lvert#1\rvert}}

\newcommand\normf[1]{\ensuremath{\lVert#1\rVert_{f}}}
\newcommand\normfRb[1]{\ensuremath{\lVert#1\rVert_{f,R_b}}}
\newcommand\normfRbone[1]{\ensuremath{\lVert#1\rVert_{f, R_{b_1}}}}
\newcommand\normfRbtwo[1]{\ensuremath{\lVert#1\rVert_{f,R_{b_2}}}}
\newcommand\normtwo[1]{\ensuremath{\lVert#1\rVert_{2}}}
\newcommand\norminfty[1]{\ensuremath{\lVert#1\rVert_{\infty}}}

\title{Almost automorphically forced flows on $S^1$ or $\mathbb{R}$  in one-dimensional almost periodic semilinear heat equations}

\author {
\\
Wenxian Shen\thanks{Partially supported by the NSF grant DMS--1645673.}\\
Department of Mathematics and Statistics\\
 Auburn University, Auburn, AL 36849, USA
\\
\\
Yi Wang\thanks{Partially supported by NSF of China No.11825106, 11771414, CAS Wu Wen-Tsun Key Laboratory of Mathematics, { University of Science and Technology of China}.} \\
School of Mathematical Science\\
 University of Science and Technology of China
\\ Hefei, Anhui, 230026, P. R. China
\\
\\
 Dun Zhou\thanks{Corresponding author, email address: zhoudun@njust.edu.cn. Partially supported by NSF of China No.11971232, 11601498 and the Fundamental Research Funds for the Central Universities No.30918011339.}\\
 School of Science
 \\ Nanjing University of Science and Technology
 \\ Nanjing, Jiangsu, 210094, P. R. China
 \\
\\
}
\date{}

\maketitle

\begin{abstract}
In this paper, we consider the asymptotic dynamics of the skew-product semiflow generated by the following {time almost-periodically forced} scalar reaction-diffusion equation
\begin{equation}
\label{eq0}
u_{t}=u_{xx}+f(t,u,u_{x}),\,\,t>0,\, 0<x<L
\end{equation}
with periodic boundary condition
\begin{equation}
\label{bdc1}
u(t,0)=u(t,L),\quad u_x(t,0)=u_x(t,L),
\end{equation}
where $f$ is uniformly almost periodic in $t$. In particular, we study the topological structure of the limit sets of { the skew-product semiflow. It is proved that
 any compact minimal invariant set (throughout this paper, we refer to it as a minimal set) can be residually embedded into an invariant set of some almost automorphically-forced flow on a circle $S^1=\mathbb{R}/L\mathbb{Z}$. Particularly, if $f(t,u,p)=f(t,u,-p)$, then the flow on a minimal set topologically conjugates to an almost periodically-forced minimal flow on $\mathbb{R}$}. Moreover, it is proved that the $\omega$-limit set of any bounded orbit contains at most two minimal sets which cannot be obtained from each other by phase translation.

 In addition, we further consider the asymptotic dynamics of the skew-product semiflow generated by \eqref{eq0} with Neumann boundary condition
 \begin{equation*}
 \label{bcd2}
 u_x(t,0)=u_x(t,L)=0,
 \end{equation*}
 or Dirichlet boundary condition
 \begin{equation*}
 \label{bdc3}
 u(t,0)=u(t,L)=0.
 \end{equation*}
 {For such system, it has been} known that the $\omega$-limit set of any bounded orbit contains at most two minimal sets.
 Under certain direct assumptions on $f$, it is proved in this paper that the flow on any minimal set of \eqref{eq0}, with Neumann boundary condition or Dirichlet boundary condition, topologically conjugates to an almost periodically-forced minimal flow on
 $\mathbb{R}$.

 Finally, a counterexample is given to show that even for quasi-periodic equations, the results we obtain here cannot be further improved in general.
\end{abstract}

\section{Introduction}

In this paper, we
consider the asymptotic dynamics of the following parabolic equation
\begin{equation}\label{equation-1}
u_{t}=u_{xx}+f(t,u,u_{x}),\,\,\,\, t>0,\,\, 0<x<L
\end{equation}
with periodic boundary condition
\begin{equation}
\label{periodic-bc}
u(t,0)=u(t,L),\quad u_x(t,0)=u_x(t,L),
\end{equation}
or Neumann boundary condition
 \begin{equation}
 \label{neumann-bc}
 u_x(t,0)=u_x(t,L)=0,
 \end{equation}
 or Dirichlet boundary condition
 \begin{equation}
 \label{dirichlet-bc}
 u(t,0)=u(t,L)=0,
 \end{equation}
where $f:\RR\times\RR\times\RR\to \RR$ and all its partial derivatives (up to order $1$) are uniformly almost periodic in $t$ (see Definition \ref{almost}).

{Recently, the dynamics of time non-periodic
	equations have been attracting more and more attention. In practical matters, large quantities of systems evolve influenced by external effects which are
	roughly, but not exactly periodic, or under environmental forcing which exhibits different,
	incommensurate periods. As a consequence, models with such time dependence are characterized
	more appropriately by quasi-periodic or almost periodic equations or even by certain nonautonomous equations rather than by periodic ones.
	
   The motivation for us to study the dynamics of equation \eqref{equation-1} is that, although the equation is typical and also is one of the simplest models of infinite-dimensional dynamical systems, the longtime behavior of  bounded solutions of \eqref{equation-1} with multi-frequency driving is far away from well understanding (for instance, $f$ of $t$ is almost periodic), especially for the system \eqref{equation-1}+\eqref{periodic-bc}.}

 The asymptotic dynamics of \eqref{equation-1} with any of the boundary conditions \eqref{periodic-bc}-\eqref{dirichlet-bc} in the case that $f$ is independent of $t$ or periodic in $t$ have been widely studied in many works and are quite well understood (see \cite{Angenent1988,BPS,CCH,Chen1989160,Chen-P,CLM1994,CRo,Fiedler,JR1,JR2,Massatt1986,Matano, Matano78,SF1,Te}).

The asymptotic dynamics of \eqref{equation-1} with Neumann { boundary condition} \eqref{neumann-bc} or Dirichlet boundary condition \eqref{dirichlet-bc} (these boundary condition cases  are also referred to as the separated boundary condition cases) was systematically studied { by Shen and Yi} in \cite{Shen1995114,ShenYi-2,ShenYi-JDDE96,ShenYi-TAMS95} in terms of the skew-product semiflow generated by \eqref{equation-1} with the corresponding boundary condition. {It is proved that the flow on any minimal set $M$ of the skew-product semiflow is an almost automorphic extension of the base flow, and if the induced flow on $M$ is unique ergodic, then it topologically conjugates to an almost periodically
 forced minimal flow on $\RR$ (see \cite{ShenYi-JDDE96}). However, it is { still unknown} whether the flow on a general minimal set topologically conjugates to an almost-periodically forced minimal flow on $\mathbb{R}$.} In this paper, we will give a confirmed answer to this question
 in the case that $f(t,u,p)=f(t,u,-p)$ when the Neumann boundary is considered and that $f(t,u,p)=-f(t,-u,p)$ or {$f(t,0,p)=0$} when the Dirichlet boundary condition is considered.

{For periodic boundary condition \eqref{periodic-bc}, observe that, by the regularity of parabolic equations (see \cite[Corollary 15.3]{Amann}), \eqref{equation-1}+\eqref{periodic-bc} can be naturally converted into the following equation on the circle $S^1$,
\begin{equation}\label{equation-circle}
u_{t}=u_{xx}+f(t,u,u_{x}),\,\,t>0,\, x\in S^1=\mathbb{R}/L\mathbb{Z}.
\end{equation}
Therefore, in this paper, we focus on \eqref{equation-circle} in stead of \eqref{equation-1}+\eqref{periodic-bc}.
The present authors in \cite{SWZ,SWZ2,SWZ3} systematically} investigated the asymptotic dynamics of { equation \eqref{equation-circle}} in the framework of the skew-product semiflows. In particular, we studied in \cite{SWZ,SWZ3} { the structures of minimal sets, as well as the $\omega$-limit sets, for the associated skew-product semiflow of equation \eqref{equation-circle}} under { some restrictions} on the dimension of their center manifolds. We { further thoroughly presented in \cite{SWZ2} the characterizations of the general $\omega$-limit sets} when $f$ in \eqref{equation-circle} satisfies the reflection symmetry condition,  i.e.,  $f(t,u,p)=f(t,u,-p)$. {One may see in Subsection \ref{dicss} more description and discussion for these results \cite{SWZ,SWZ2,SWZ3}.} {Nevertheless, it remains unknown for} the universal dynamics on a general minimal set of  the skew-product semiflow generated by \eqref{equation-circle},
{ as well as the structure of a general $\omega$-limit set} of the skew-product semiflow. In this paper, { among others, we will present the structural theorems for the minimal sets and the $\omega$-limit sets}, without { any restriction} on the dimension of the center manifolds, which are major improvements of the works \cite{SWZ,SWZ3}.

In the following {two subsections}, we state our main results, {and give series of remarks on these main results} of this paper.

\subsection{Statements of the main results}

In this subsection, we state the main results of this paper.
We first consider \eqref{equation-circle} and introduce the skew-product semiflow generated by \eqref{equation-circle}. Throughout this paper, we assume that $f(t,u,p)\in C^1(\mathbb{R}\times \mathbb{R} \times \mathbb{R},\mathbb{R})$ and that  $f$ and  all its partial derivatives (up to order $1$) are uniformly almost periodic in $t$ (see Definition \ref{almost}). Then $f_{\tau}(t,u,p)=f(t+\tau,u,p)\, (\tau \in \RR)$ generates a family
$\{f_{\tau}|\tau \in \mathbb{R}\}$ in the space of continuous functions $C(\mathbb{R}\times \mathbb{R} \times \mathbb{R},\mathbb{R})$ equipped with the compact open topology. The closure $H(f)$ of $\{f_{\tau}|\tau\in \mathbb{R}\}$ in
the compact open topology, called the hull of $f$, is a compact metric space and every $g\in H(f)$ has the same regularity as $f$. Thus, the time-translation $g\cdot t\equiv g_{t}\,(g\in H(f))$ defines a compact flow on $H(f)$. Note that the flow on $H(f)$ is  minimal, that is, it is the only nonempty compact subset of itself that is invariant under the flow $g\cdot t$.

Equation \eqref{equation-circle} naturally
induces a family of equations associated to each $g\in H(f)$,
\begin{equation}\label{equation-lim1}
u_{t}=u_{xx}+g(t,u,u_{x}),\,\,\quad t> 0,\quad \, S^1=\mathbb{R}/L\mathbb{Z}.
\end{equation}
Assume that $X$ is a fractional power space associated with the operator $u\rightarrow
-u_{xx}:H^{2}(S^1)\rightarrow L^{2}(S^1)$ satisfies $X\hookrightarrow C^{1}(S^1)$ (that is, $X$ is compactly embedded into $C^{1}(S^1)$). For any $u\in X$, \eqref{equation-lim1} defines (locally) a unique solution $\varphi(t,\cdot;u,g)$ in $X$ with $\varphi(0,\cdot;u,g)=u(\cdot)$ and it continuously depends on $g\in H(f)$ and $u\in
X$. Consequently, \eqref{equation-lim1} admits a (local) skew-product semiflow $\Pi_{t}$ on $X\times
H(f)$:
\begin{equation}\label{equation-lim2}
\Pi_{t}(u,g)=(\varphi(t,\cdot;u,g),g\cdot t),\quad t\ge 0.
\end{equation}
It follows from \cite{Hen} (see also \cite{Hess,Mierczynski}) and the standard a priori estimates for parabolic equations that,  if
$\varphi(t,\cdot;u,g) (u\in X)$ is bounded in $X$ in the existence interval of the solution, then $u$ is a globally defined classical solution. Moreover,  for any $\delta>0$, $\{\varphi(t,\cdot;u,g):t\ge \delta\}$ is relatively compact in $X$. Consequently, the $\omega$-limit set $\omega(u,g)$ of the bounded semi-orbit $\Pi_{t}(u,g)$ in $X\times H(f)$ is a nonempty connected compact subset of $X\times H(f)$. { We write $p:X\times H(f)\rightarrow H(f)$ as the nature projection onto $H(f)$.}

\bigskip

 Given any $u\in X$ and $a\in S^1$, we define the shift $\sigma_a$ on $u$ as $(\sigma_a u)(\cdot)=u(\cdot+a)$. Let $u\in A\subset X$, we write
 \begin{equation*}\label{E:group-orbit-11}
 \Sigma u=\{\sigma_a u\,|\, a\in S^1\}
 \end{equation*} as the {\it $S^1$-group orbit} of $u$, and write $\sigma_a A=\{\sigma_au|u\in A\}$ and $\Sigma A=\cup_{u\in A}\Sigma u$, respectively. { A subset $M\subset X\times H(f)$ is called {\it spatially-inhomogeneous}, if any $u\in M$ is independent of spatial variable $x$.}

 { A series of theorems we are about to state are strongly dependent on the following important technical Lemma, which
 	exhibits the constancy property of zero number on the minimal set.
 	
 	\begin{lemma}
 		\label{center-constant}
 		(\textbf{Constancy Property Lemma}) Consider \eqref{equation-circle}. Let $M\subset X\times H(f)$ be a minimal set of \eqref{equation-lim2}. Then
 		there exists $N\in \mathbb{N}$ such that
 		\begin{equation}\label{const-mini}
 		z(\varphi(t,\cdot;u_1,g)-\varphi(t,\cdot;\sigma_au_2,g))=N
 		\end{equation}
 		for any $t\in\mathbb{R}$, any $(u_1,g), (u_2,g)\in M\cap p^{-1}(g)$, and any  $a\in S^1 \text{ with }u_1\neq \sigma_au_2$, $z(\cdot)$ is the zero number function defined in Section 2.3.
 \end{lemma}}

 \vskip 1mm
 { Our first two theorems indicate that, for the skew-product semiflow \eqref{equation-lim2} generated by equation \eqref{equation-circle}, any spatially-inhomogeneous minimal set $M\subset X\times H(f)$ can be residually embedded into an invariant set of an almost automorphically-forced circle flow. If, in addition, $f(t,u,u_x)=f(t,u,-u_x)$ in \eqref{equation-circle}, then the flow on $M$ topologically conjugates to an almost periodically-forced minimal flow on $\mathbb{R}$.}

{
\begin{theorem}\label{a-a-circle-flow}
Let $M\subset X\times H(f)$ be a spatially-inhomogeneous  minimal set of \eqref{equation-lim2}.
Then
\begin{itemize}

\item[(1)] There exists $L^0\in (0,L]$ such that each $u\in M$ shares a common smallest spatial-period $L^0$, i.e., $u(\cdot+L^0)=u(\cdot)$ and $u(\cdot+a)\neq u(\cdot)$ for any $a\in (0,L^0)$.

 \item[(2)]There is a residual invariant set $Y_0\subset H(f)$ such that, for any $g\in Y_0$, there exists $u_g \in X$ satisfying $u_g(0)=\max_{x\in S^1}u_g(x)$ and $p^{-1}(g)\cap M\subset (\Sigma u_g,g)$.
  \end{itemize}
Assume further that $f$ is $C^2$-admissible (see Definition \ref{admissible}) and let $M_0=\{(u,g)\in M\,|\, g\in Y_0\}$. Then
\begin{itemize}
 \item[(3)]
For any given $g\in Y_0$, the function
$t\mapsto G(t;g)=g_p(t,u_{g\cdot t}(0),0)+\frac{u^{'''}_{g\cdot t}(0)}{u^{''}_{g\cdot t}(0)}
$ is alomst-automorphic. Let
$$
\bar M_0=\{(c(u),G(\cdot;g))\,|\, (u,g)\in M_0,\,\, c(u)\in \mathcal{S} \,\, \text{is such that}\,\, u(\cdot)=u_g(\cdot+c(u))\},
$$
where  $\mathcal{S}=\mathbb{R}/L^0\mathbb{Z}$. Then
$$\bar h_0:M_0\to \bar M_0;\,\, \bar h_0(u,g)=(c(u),G(\cdot;g))\subset \mathcal{S}\times H(G)
$$
is a homeomorphism.

\item[(4)] For any given $g\in Y_0$, there is a $C^1$-function $c^g:\mathbb{R}\to \mathcal{S}, t\mapsto c^g(t)$ (with its derivative $\dot c^g(t)$ being almost-automorphic in $t$) such that
$$
{\bar \Pi_t^0: \bar M_0\to\bar M_0;\quad \bar\Pi_t^0 (c,G)=(c+c^g(t),G(\cdot+t,g))}
$$
is a flow on $\bar M_0$, and
 $$
\bar h_0\Pi_t (u,g)=\bar\Pi_t^0 \bar h_0(u,g)\quad \forall \, t\in \mathbb{R},\, \, (u,g)\in M_0.
$$
That is, the following diagram is commutative:
$$
\begin{CD} M_0 @>\bar h_0>> \bar{M}_0\\ @VV\Pi_tV @VV\bar\Pi_t^0V\\ M_0 @>\bar h_0>> \bar{M}_0.
\end{CD}$$
  \end{itemize}
\end{theorem}
}

 \begin{theorem}\label{a-p-flow}
{ Assume that $f(t,u,u_x)=f(t,u,-u_x)$ in \eqref{equation-circle}.} Let $M\subset X\times H(f)$ be a minimal set of \eqref{equation-lim2}. Then,
there is $x_0\in S^1$ such that the mapping
 $$\bar h: M\to \bar M:=\{(u(x_0),g)\,|\, (u,g)\in M\},\quad h(u(\cdot),g)=(u(x_0),g),
   $$
   is a homeomorphism;
   $$\bar\Pi_t: \bar M\to\bar M,\quad \bar\Pi_t(u(x_0),g)=(\phi(t,x_0;u,g),g\cdot t),
   $$
    is a minimal flow { satisfying}
  $\bar h(\Pi_t (u,g))= \bar\Pi_t \bar h(u,g)$.
\end{theorem}

{ The following Theorem \ref{structure-thm} gives a trichotomy of the structure of a general $\omega$-limit set of \eqref{equation-lim2}.}

\begin{theorem}\label{structure-thm}
 Assume that $\Omega$ is an $\omega$-limit set of \eqref{equation-lim2}.
Then one of the following alternatives must hold:
\begin{itemize}
\item[{ \rm (1)}] There is a minimal set $M\subset \Omega$ such that $\Omega\subset \Sigma M$.

\item[{ \rm (2)}] There is a minimal set $M_1\subset \Omega$ such that $\Omega\subset \Sigma M_1\cup M_{11}$, where $M_{11}\subset \Omega$ with $M_{11}\neq \emptyset$ and $M_{11}$ connects $\Sigma M_1$ in the sense that if $(u_{11},g)\in M_{11}$, then $\Sigma M_1\cap \omega(u_{11},g)\not =\emptyset$ and $\Sigma M_1\cap\alpha (u_{11},g)\not =\emptyset$.

\item[{ \rm (3)}] There are two minimal sets $M_1,M_2\subset \Omega$ with $\Sigma M_1\cap \Sigma M_2=\emptyset$ such that
 $\Omega\subset \Sigma M_1\cup \Sigma M_2\cup M_{12}$, where $M_{12}\subset \Omega$ with $M_{12}\not =\emptyset$, and for any $(u_{12},g)\in M_{12}$, $\omega(u_{12},g)\cap (\Sigma M_1\cup\Sigma M_2)\not=\emptyset$ and  $\alpha(u_{12},g)\cap (\Sigma M_1\cup\Sigma M_2)\not=\emptyset$.
\end{itemize}
\end{theorem}

To state the main results on the asymptotic dynamics of \eqref{equation-1} with Neumann boundary condition \eqref{neumann-bc}
or Dirichlet boundary condition \eqref{dirichlet-bc}, we introduce the following standing assumptions on $f(t,u,p)$.

\smallskip

\noindent {\bf (HNB)} $f(t,u,-p)=f(t,u,p)$ for any $(t,u,p)\in \RR\times\RR\times\RR$.

\smallskip

\noindent {\bf (HDB1)} $f(t,-u,p)=-f(t,u,p)$ for any $(t,u,p)\in \RR\times\RR\times\RR$.
\smallskip

\noindent {\bf (HDB2)}  $f(t,0,p)\equiv 0$.

\smallskip

Consider \eqref{equation-1}+\eqref{neumann-bc}. Let $X_N$ be a fractional power space associated with the operator $u\rightarrow -u_{xx}:H^{2}(0,L)\rightarrow L^{2}(0,L)$ with Neumann boundary condition $u_x(0)=u_x(L)=0$ such that $X_N\hookrightarrow C^{1}[0,L]$.
 Let $\Pi^N_t$ be the  (local) skew-product semiflow on $X_N\times
H(f)$ generated by \eqref{equation-1}+\eqref{neumann-bc},
\begin{equation}
\label{Ne-skew-product}
\Pi_t^N(u_0,g)=(\phi^N(t,x;u_0,g),g\cdot t)\quad \forall\,\, u_0\in X_N,\,\, g\in H(f),
\end{equation}
where $g\in H(f)$ and $u(t,x)=\phi^N(t,x;u_0,g)$ is the solution of
\begin{equation}
\label{neumann-eq-bc}
\begin{cases}
u_t=u_{xx}+g(t,u,u_x),\quad 0<x<L\cr
u_x(t,0)=u_x(t,L)=0
\end{cases}
\end{equation}
 with $u(0,x)=u_0(x)$. For a given minimal set $M\subset X_N\times H(f)$ of $\Pi_t^N$, let
$$
M_0^N=\{(u_0(0),g)\,|\, (u_0,g)\in M\}.
$$
We have

 \begin{theorem}
 \label{Ne-imbed} Consider {\rm \eqref{equation-1}+\eqref{neumann-bc}} and assume
{\bf (HNB)}. Then for any minimal set $M$ of $\Pi_t^N$, the mapping $M\ni(u_0,g)\to (u_0(0),g)\in M_0^N$ is a continuous bijection,
and $\Pi_t^N|_M$ is topologically conjugated to $\pi_t^N:M_0^N\to M_0^N$, where
$$
\pi_t^N(u_0(0),g)=(\phi^N(t,0;u_0,g),g\cdot t)\quad \forall\, (u_0,g)\in M.
$$
Hence,  the flow on $M$ topologically conjugates to an almost periodically-forced minimal flow on $\mathbb{R}$.
 \end{theorem}

Consider \eqref{equation-1} +\eqref{dirichlet-bc}. Let $X_D$ be a fractional power space associated with the operator $u\rightarrow -u_{xx}:H^{2}(0,L)\rightarrow L^{2}(0,L)$ with Dirichlet boundary condition $u(0)=u(L)=0$ such that $X_D\hookrightarrow C^{1}[0,L]$.
 Let $\Pi^D_t$ be the  (local) skew-product semiflow on $X_N\times
H(f)$ generated by \eqref{equation-1} +\eqref{dirichlet-bc},
\begin{equation}
\label{Di-skew-product}
\Pi_t^D(u_0,g)=(\phi^D(t,x;u_0,g),g\cdot t)\quad \forall\,\, u_0\in X_N,\,\, g\in H(f),
\end{equation}
where $g\in H(f)$ and $u(t,x)=\phi^D(t,x;u_0,g)$ is the solution of
\begin{equation}
\label{dirichlet-eq-bc}
\begin{cases}
u_t=u_{xx}+g(t,u,u_x),\quad 0<x<L\cr
u(t,0)=u(t,L)=0
\end{cases}
\end{equation}
 with $u(0,x)=u_0(x)$. For a given minimal set $M\subset X_D\times H(f)$, let
 $$
 M_0^D=\{(u_0^{'}(0),g)\,|\, (u_0,g)\in M\}.
 $$
 We have

 \begin{theorem}
 \label{Di-imbed}  Consider \eqref{equation-1} +\eqref{dirichlet-bc}.  Let $M$ be a minimal set of $\Pi_t^D$ and
 $$
 M_0^D=\{(u_0^{'}(0),g)\,|\, (u_0,g)\in M\}.
 $$
 Assume that {\bf (HDB1)} holds or that {\bf (HDB2)} holds and $u_0\ge 0$ for any $(u_0,g)\in M$. Then
 the mapping $M\ni(u_0,g)\to (u_0^{'}(0),g)\in M_0^D$ is a continuous bijection,
and $\Pi_t^D|_M$ is {topologically} conjugated to $\pi_t^D:M_0^D\to M_0^D$, where
$$
\pi_t^D(u_0^{'}(0),g)=(\phi^D_x(t,0;u_0,g),g\cdot t)\quad \forall\, (u_0,g)\in M.
$$
Hence, the flow on $M$ conjugates to an almost periodically-forced minimal flow on $\mathbb{R}$.
 \end{theorem}

\subsection{Remarks on the main results}\label{dicss}

In this subsection, we give some remarks on the above results obtained in this paper.

\smallskip

1) For equation \eqref{equation-1} with periodic boundary condition (as it stated before, equation \eqref{equation-1} can be converted into \eqref{equation-circle}), Lemma \ref{center-constant} and Theorems \ref{a-a-circle-flow}-\ref{structure-thm} extend most of our results in the previous works  \cite{SWZ,SWZ2,SWZ3} to general cases. As a matter of fact, our series of works have formal correspondences with respect to those in \cite{Shen1995114,ShenYi-2,ShenYi-TAMS95,ShenYi-JDDE96,Shen1998} for separated boundary conditions. For instance, the property that a minimal set is residually embedded into an invariant set of some almost automorphically forced circle flow corresponds to the property that a minimal set is an almost $1$-cover of the base flow in separated boundary condition cases (\cite[Theorem 2.6]{ShenYi-TAMS95}); and the trichotomy property of $\omega$-limit set resembles in the separated condition cases that $\omega$-limit set contains at most two minimal sets(\cite[Theorem 2.6]{Shen1995114}). However, the limit sets we consider here are more complicated than those in separated boundary condition cases, because the almost automorphically (periodically)-forced circle flows can be very complicated (see \cite{HuYi} and the literatures therein).

\smallskip

2) The constancy property of zero number on the minimal set $M$ (see Lemma \ref{center-constant}) is the cornerstone of our entire article, { since our main results are highly dependent on this property}. In our previous works \cite{SWZ,SWZ3} {related to the periodic boundary condition case, to ensure this constancy, we { need to assume} that the dimension of center space of minimal set is no more than $2$, some times one also needs to assume that the dimension of unstable manifold is odd. { More precisely, let $V^c(M)$, $V^u(M)$ be the associated center space and unstable space of $M$ (associated with Sacker-Sell bundles), respectively.
In \cite{SWZ}, we proved Theorem \ref{a-a-circle-flow} via the invariant manifold approach under the assumption that $\dim V^c(M)\leq 2$; and if $\dim V^c(M)=2$, we even required $\dim V^u(M)$ being odd. As a consequence, in \cite{SWZ2}, the structural Theorem \ref{structure-thm} for $\omega$-limit set $\Omega$ was proved under the analogous assumptions.} { Undoubtedly}, there are some limitations of these previous works. As a matter of fact, for a given { $\omega$-limit set or even a minimal set}, these dimension assumptions are not easy to check; { and moreover, the structure of a general $\omega$-limit or a minimal set is still far away from well understanding. To investigate these structures, it appears to us that our invariant manifolds approach (associated with the Sacker-Sell bundles) in \cite{SWZ,SWZ2,SWZ3} is sort of special and not fine enough (see \cite{P.Bates,Chow1994,Chow1991,Chow1994-2,Sacker1978,Sacker1991} for Sacker-Sell spectrum and related invariant manifolds theories of parabolic equations). Thus, to improve our previous works and solve these problems, one needs to find new ways. Fortunately, in the current paper, these qualifications can be removed.} The invariance of $S^1$-group action and the corresponding relationship between boundedness of zero number function along the nontrivial solution of \eqref{linear-equation2} and the exponential boundedness of this solution (see Lemma \ref{zero-up-control}) can help us  reach the constancy property of zero number, which in fact related to the Floquet bundles theories obtained in \cite{Chow1995}.

{ Theorem \ref{a-a-circle-flow} reveals} { that residually-embedding into an invariant set of almost automorphically-forced circle flow is an universal property} of minimal sets of \eqref{equation-lim2}. Certainly, if one { wishes to obtain a more delicate result for the structure of the $\omega$-limit set $\Omega$ (like embedding $\Omega$ into a compact invariant set of some almost periodically-forced circle flow), some restrictions { on} the dimension of center space of $\Omega$ are needed} ({ see \cite{SWZ,SWZ3} for detailed discussion). In fact, a counter-example has been provided in \cite{SWZ3}} that an $\omega$-limit set with { $2$-dimensional} center space cannot be embedded { into} a compact invariant set of almost periodically-forced circle flow.

\smallskip

3) { If in addition $f(t,u,u_x)=f(t,u,-u_x)$ (for example, $f=f(t,u)$), it was shown in \cite{SWZ2} that any minimal set $M$ is an almost $1$-cover of $H(f)$ (see Definition \ref{1-cover} for almost $1$-cover). In Theorem \ref{a-p-flow}, we proved that all points in $M$ on the same base $g\in H(f)$ are ordered; and hence, the minimal flow generated by \eqref{equation-lim2} topologically conjugates to an almost periodically forced minimal flow on $\mathbb{R}$, which is stronger than the result in \cite{SWZ2}. It { entails} that the flow on $M$ can be viewed as a minimal flow generated by an almost periodically-forced 1-dimensional ordinary differential equation, and certainly as an almost $1$-cover of $H(f)$ (see \cite[Remark 3.4]{ShenYi-TAMS95}).}

\smallskip

4) Comparing with { equation} \eqref{equation-circle} in time independent or periodic case, the asymptotic dynamics with time almost periodic dependence are more complicated. { As a matter of fact}, in autonomous system, any periodic orbit is a rotating wave $u=\phi(x-ct)$ for some $2\pi$-periodic function $\phi$ and constant $c$; and hence, $\omega(u)$ is either itself a single rotating wave, or a set of equilibria differing only by phase shift in $x$, see Massatt \cite{Massatt1986} and Matano \cite{Matano}. While, for time-periodic { system}, Sandstede and Fiedler \cite{SF1} showed that the  $\omega$-limit set $\omega(u)$ can be viewed as a subset of the two-dimensional torus carrying a linear flow. { In other words}, only case (i) in Theorem \ref{structure-thm} will happen for both autonomous and periodic systems, { that is}, the whole $\omega$-limit set can be embedded into a compact invariant set of (resp. periodically-forced) circle flow for autonomous (resp. periodic) case. However, in our current almost periodic systems, without further assumptions, $\omega$-limit set may not be embedded into a compact invariant set of some almost periodically-forced circle flow. Indeed, all the three cases in Theorem \ref{structure-thm} can happen (see examples in \cite{SWZ2}). All these facts reveal that the { almost periodically-forced} non-autonomous systems we consider here are essentially different from autonomous and periodic systems.

\smallskip
5) For \eqref{equation-1} with separated boundary conditions cases, it is proved { by Shen and Yi  \cite{ShenYi-JDDE96}} that, if the induced flow on a minimal set of the corresponding skew-product semiflow is unique ergodic, then it topologically conjugates to an almost periodically
 forced minimal flow on $\RR$.   In this paper,  under the assumption {\bf (HNB)}, or {\bf (HDB1)}, or {\bf (HNB2)},
 we prove that a general minimal set of the corresponding skew-product semiflow topologically conjugates to an almost periodically
 forced minimal flow on $\RR$. Note that
 the assumption {\bf (HNB)}, or {\bf (HDB1)}, or {\bf (HNB2)} is on the equation { itself and can be easily checked}.
\smallskip

6) Although the conclusions in this paper are on time almost-periodic systems, most of our conclusions can still be established for more general time dependent function $f$, for example, $f$  is recurrent in $t$. Indeed, { the Constancy Property Lemma} and Theorem \ref{structure-thm} { remain valid} since { their} proofs only require { the minimality of} $H(f)$, and Theorems \ref{a-p-flow}, \ref{Ne-imbed}, \ref{Di-imbed} can still be established if { one} removed ``almost periodically-forced'' { statement} in these theorems. Additionally, we give an example { at the end this paper} to remark that even for $f$ being is quasi-periodic in $t$, Theorem \ref{a-a-circle-flow} can not be { improved}.

\smallskip

7) We remark that if $f=f(t,x,u,u_x)$ in { equation} \eqref{equation-circle} depends on $x$, the asymptotic behavior of solutions of the time almost-periodic system is far from being studied. The structure of a minimal set was only given under stability assumptions (see \cite{SWZ}). The reason is due to the difficulty in zero number control on the minimal set and the lacking of invariance of $S^1$-group action. In fact, Sandstede and Fiedler in \cite{SF1} had already pointed out that chaotic behavior exhibited by any time-periodic planar vector field can also be found in certain time-periodic equation with the nonlinearity $f=f(t,x,u,u_x)$, not to mention the complexity of dynamics of almost-periodic systems.

\bigskip

The rest of this paper is organized as follows. In section 2, we list some conceptions of compact dynamical systems, almost-periodic (almost-automorphic) functions, and introduce some properties of zero number function of the linearized system associated with equation \eqref{equation-circle}. In section 3, {we prove our main results related to the periodic boundary condition case. In section 4, we obtain that for Neumann or Dirichlet boundary condition with assumption {\bf (HNB)}, or {\bf (HDB1)}, or {\bf (HNB2)}, the flow on the minimal set also topologically conjugates to an almost periodically-forced minimal flow on $\mathbb{R}$}. In the last section, we give an example of quasi-periodic ordinary equation on a circle to illustrate that the results obtained in the current paper can not be improved for general quasi-periodic equations.

\section{Preliminaries}

In this section, we introduce some concepts, notations and properties which will be often used in the next section.

\subsection{Some concepts of compact dynamical systems}\label{ScDS}
Let $Y$ be a compact metric space with metric $d_{Y}$, and
$\sigma:Y\times \RR\to Y, (y,t)\mapsto y\cdot t$ be a continuous
flow on $Y$, denoted by $(Y,\sigma)$ or $(Y,\RR)$. {A subset
$E\subset Y$ is {\it invariant} if $\sigma_t(E)=E$ for every $t\in
\RR$. A subset $E\subset Y$ is called {\it minimal} if it is
compact, invariant and the only non-empty compact invariant subset
of it is itself.  It is known that every compact and $\sigma$-invariant set contains a minimal subset and a subset $E$ is minimal if and only if every trajectory is
 dense. The continuous flow $(Y,\sigma)$ is called
{\it recurrent} or {\it minimal} if $Y$ is minimal.} A pair
$y_1,y_2$ of different elements of $Y$ are said to be {\it positively proximal} (resp. {\it negatively proximal}), if there is $t_n\to\infty$ (resp. $t_n\to-\infty$) as $n\to\infty$ such
that $d_Y(y_1\cdot t_n,y_2\cdot t_n)\to 0$, the pair $y_1,y_2$ is called {\it two sided proximal} if it is both a positively and negatively proximal pair.

\begin{definition}
\label{1-cover}
{\rm
Let $(Y,\mathbb{R})$, $(Z,\mathbb{R})$ be two continuous compact flows. $Z$ is called a {\it $1$-cover} ({\it almost $1$-cover}) of $Y$ if there is a surjective flow homomorphism $p:Z\to Y$ such that $p^{-1}(y)$ is a singleton for any $y\in Y$ (for at least one $y\in Y$).}
\end{definition}

\begin{lemma}\label{epimorphism-thm}
Let $p:(Z,\mathbb{R})\rightarrow(Y,\mathbb{R})$ be an epimorphism of flows, where $Z,Y$ are compact
metric spaces. Then the set
\begin{center}
$Y'=\{y_{0}\in Y: for \ any\ z_{0}\in p^{-1}(y_{0}),\ y\in Y \ and\ any\ sequence\
\{t_{i}\}\subset \RR \ with\ y\cdot t_{i}\rightarrow y_{0},\ there\ is\ a\ sequence\
\{z_{i}\}\in p^{-1}(y)\ such\ that\ z_{i}\cdot t_{i}\rightarrow z_{0} \}$
\end{center}
\vskip -2mm
is residual and invariant. In particular, if $(Z,\mathbb{R})$ is minimal and distal, then $Y'=Y$.
\end{lemma}
\begin{proof}
See \cite[Lemma I.2.16]{Shen1998} and the remarks below it.
\end{proof}

Let $X,Y$ be metric spaces and $(Y,\sigma)$ be a compact flow (called the base flow). Let also
 $\RR^+=\{t\in \RR:t\ge 0\}$. A
 {\it skew-product semiflow} $\Pi_{t}:X\times Y\rightarrow X\times Y$ is a semiflow
 of the following form
 \begin{equation}\label{skew-product-semiflow}
 \Pi_{t}(u,y)=(\varphi(t,u,y),y\cdot t),\quad t\geq 0,\, (u,y)\in X\times Y,
 \end{equation}
satisfying (i) $\Pi_{0}={\rm Id}_X$ and (ii)
$\varphi(t+s,u,y)=\varphi(s,\varphi(t,u,y),y\cdot t)$ for each $(u,y)\in X\times Y$ and $s,t\in \RR^+$.
 A subset $A\subset
X\times Y$ is {\it positively invariant} if $\Pi_{t}(A)\subset A$ for
all $t\in \RR^+$. The {\it forward orbit} of any $(u,y)\in X\times
Y$ is defined by $\mathcal{O}^+(u,y)=\{\Pi_{t}(u,y):t\ge 0\}$, and the
{\it $\omega$-limit set} of $(u,y)$ is defined by
$\omega(u,y)=\{(\hat{u},\hat{y})\in X\times Y:\Pi^{t_n}(u,y)\to
(\hat{u},\hat{y}) (n\to \infty) \textnormal{ for some sequence
}t_n\to \infty\}$.

A {\it flow extension} of a skew-product semiflow $\Pi_{t}$
 is a continuous skew-product flow $\hat{\Pi}_t$ such that $\hat{\Pi}_t(u,y)=\Pi_{t}(u,y)$ for each
$(u,y)\in X\times Y$ and $t\in \RR^+$. A compact positively
invariant subset is said to admit {\it a flow extension} if the
semiflow restricted to it does. Actually, a compact positively
invariant set $K\subset X\times Y$ admits a flow extension if every
point in $K$ admits a unique backward orbit which remains inside the
set $K$ (see \cite[part II]{Shen1998}). A compact positively
invariant set $K\subset X\times Y$ for $\Pi_{t}$ is called {\it minimal} if it
does not contain any other nonempty compact positively invariant set
than itself.

\begin{definition}
\label{conjugate-def}
  {\rm
  Two flows $(Y,\mathbb{R})$ and $(Z,\mathbb{R})$ are said to be {\it topologically conjugate} if there is a homeomorphism $h:Y\to Z$ such that $h(y\cdot t)=h(y)\cdot t$ for all $y\in Y$ and $t\in \mathbb{R}$.}
\end{definition}
\begin{definition}\label{residual-imbed}
{\rm Let $X_1$ be a metric space. A minimal subset $M\subset X\times Y$ (here $X$, $Y$ are metric spaces) is said to be {\it residually
embedded} into $X_1\times Y$, if there exist a residual invariant set $Y_*\subset Y$ and a flow $\hat{\Pi}_t$ defined on some subset $\hat{M}\subset X_1\times Y_*$ such that the flow $\Pi_{t}|_{M\cap p^{-1}(Y_*)}$ is topologically conjugate to $\hat{\Pi}_t$  on ${\hat{M}}$.
Moreover, when $Y_*=Y$, we call $M$ is {\it embedded} into $X_1\times Y$.
}\end{definition}

\subsection{Almost periodic (automorphic) functions (minimal flows)}

Assume $D$ is a nonempty subset of $\RR^m$. We list some necessary definitions and properties related to almost periodic (automorphic) functions in the following.
\begin{definition}\label{admissible}
A function $f\in C(\RR\times D,\RR)$ is said to be {\rm  admissible} if for any compact subset $K\subset D$, $f$ is bounded and uniformly continuous on
$\RR\times K$. $f$ is {\rm $C^r$ ($r\ge 1$)  admissible} if $f(t,w)$ is $C^r$ in $w\in D$ and Lipschitz in $t$, and $f$ as well as its partial derivatives to order $r$ are admissible.
\end{definition}

Let $f\in C(\RR\times D,\RR)$ be an admissible function. Then
$H(f)={\rm cl}\{f\cdot\tau:\tau\in \RR\}$ (called the {\it hull of
$f$}) is compact and metrizable under the compact open topology (see \cite{Sell,Shen1998}), where $f\cdot\tau(t,\cdot)=f(t+\tau,\cdot)$. Moreover, the time translation $g\cdot t$ of $g\in H(f)$ induces a natural
flow on $H(f)$ (cf. \cite{Sell}).

\begin{definition}\label{almost}
{\rm
\begin{itemize}

\item[(1)] A function $f\in C(\RR,\RR)$ is {\it almost periodic} if for every $\{t'_k\}\subset\mathbb{R}$ there is a subsequence $\{t_k\}\subset \{t'_k\}$
such that $\{f(t+t_k)\}$ converges uniformly.

\item[(2)] $f\in C(\RR,\RR)$ is {\it almost automorphic} if for every $\{t'_k\}\subset\mathbb{R}$ there is a subsequence $\{t_k\}\subset \{t'_k\}$
and a function $g:\mathbb{R}\to \mathbb{R}$ such that $f(t+t_k)\to g(t)$ and $g(t-t_k)\to f(t)$ pointwise.

\item[(3)]  A function $f\in C(\RR\times D,\RR)(D\subset \RR^m)$ is {\it  uniformly almost periodic in $t$} (resp. {\it uniformly almost automorphic in $t$}), if $f$ is both admissible and, for each fixed $d\in D$, $f(t,d)$ is  almost periodic (resp. almost automorphic)  with respect to $t\in \RR$.
\end{itemize}
}
\end{definition}

\begin{remark}\label{a-p-to-minial}
{\rm  If $f$ is a uniformly almost periodic (automorphic) function in $t$, then $H(f)$ is always minimal, and $(H(f),\mathbb{R})$ is an almost periodic (automorphic) minimal flow.
 Moreover, $g$ is a uniformly almost periodic (automorphic) function for all (residually many)
$g\in H(f)$ (see, e.g. \cite{Shen1998}).}
\end{remark}

\subsection{Zero number function for parabolic equations on $S^1$}
Given a $C^{1}$-smooth function $u:S^{1}\rightarrow \mathbb{R}$, the zero number of $u$ is
 defined as
$$z(u(\cdot))={\rm card}\{x\in S^{1}|u(x)=0\}.$$
The following key lemma describes the behavior of the zero number for linear non-autonomous parabolic equations, which was originally presented in \cite{2038390,H.MATANO:1982} and improved in \cite{Chen98}.
\begin{lemma}\label{zero-number}
Consider the linear system
\begin{equation}\label{linear-equation}
\begin{cases}
\varphi_{t}=\varphi_{xx}+b \varphi_{x}+c\varphi,\quad x\in S^1,\\
\varphi_{0}=\varphi(0,\cdot)\in H^{1}(S^{1}),
\end{cases}
\end{equation}
where the coefficients $b,c$ are allowed to depend on $t$ and $x$ such that
$b,b_{t},b_{x},c\in L_{loc}^\infty$. Let $\varphi(t,x)$ be a nontrivial solution of \eqref{linear-equation}.
Then the following properties holds.\par
{\rm (a)} $z(\varphi(t,\cdot))<\infty,\forall t>0$ and is non-increasing in t.\par
{\rm (b)} $z(\varphi(t,\cdot))$ drops at $t_{0}$ if, and only if, $\varphi(t_{0},\cdot)$ has a
multiple zero in $S^{1}$.\par
{\rm (c)}  $z(\varphi(t,\cdot))$ can drop only finite many times, and there exists a $T>0$
such that $\varphi(t,\cdot)$ has only simple zeros in $S^{1}$ as $t\geq T$(hence
$z(\varphi(t,\cdot))=constant$ as $t\geq T$).
\end{lemma}
Immediately, we have the following
\begin{lemma}\label{difference-lapnumber}
For any $g\in H(f)$, let $\varphi(t,\cdot;u,g)$ and $\varphi(t,\cdot;\hat{u},g)$ be two
distinct solutions of {\rm (\ref{equation-lim1})} on
$ \mathbb{R}^+$. Then

{\rm (a)} $z(\varphi(t,\cdot;u,g)-\varphi(t,\cdot;\hat{u},g))<\infty$ for $t>0$ and is non-increasing in t;

{\rm (b)} $z(\varphi(t,\cdot;u,g)-\varphi(t,\cdot;\hat{u},g))$
strictly decreases at $t_0$ if, and only if, $\varphi(t_0,\cdot;u,g)-\varphi(t_0,\cdot;\hat{u},g)$ has a multiple
zero on $S^1$;

{\rm (c)} $z(\varphi(t,\cdot;u,g)-\varphi(t,\cdot;\hat{u},g))$ can drop only finite many times, and there exists a $T>0$ such that  $$z(\varphi(t,\cdot;u,g)-\varphi(t,\cdot;\hat{u},g))\equiv
\textnormal{constant}$$ for all $t\ge T$.

\end{lemma}

\begin{lemma}\label{zero-cons-local}
Let $u\in X$ be such that $u$ has only simple zeros on $S^1$, then there exists a $\delta>0$ such that for any $v\in X$ with $\|v\|<\delta$, one has
  \begin{equation*}
    z(u)=z(u+v).
  \end{equation*}
\end{lemma}
\begin{proof}
See Corollary 2.1 in \cite{SF1} or Lemma 2.3 in \cite{Chen1989160}.
\end{proof}

The proof of the following lemma can be found in \cite[Lemma 2.5]{SWZ}.
\begin{lemma}\label{sequence-limit}
Fix $g,\ g_{0}\in H(f)$. Let $(u^{i},g)\in p^{-1}(g),(u_{0}^{i},g_{0})\in p^{-1}(g_{0})$  $ (i=1,\ 2,\ u^{1}\neq u^{2},\ u_{0}^{1}\neq u_{0}^{2})$ be such that $\Pi_{t}(u^{i},g)$ is defined on $\mathbb{R}^{+}$ (resp. $\mathbb{R}^-$) and $\Pi_{t}(u_{0}^{i},g_{0})$ is defined on $\mathbb{R}$. If there exists a sequence $t_{n}\rightarrow +\infty$ (resp. $s_{n}\rightarrow -\infty$) as $n\rightarrow \infty$, such that $\Pi_{t_{n}}(u^{i},g)\rightarrow (u_{0}^{i},g_{0})$ (resp. $\Pi_{s_{n}}(u^{i},g)\rightarrow (u_{0}^{i},g_{0})$) as $n\rightarrow \infty (i=1,2)$, then
$$z(\varphi(t,\cdot;u_{0}^{1},g_{0})-\varphi(t,\cdot;u_{0}^{2},g_{0}))\equiv \textnormal{constant},$$
for all $t\in \mathbb{R}$.
\end{lemma}

Consider the following linear parabolic equation:
\begin{equation}\label{linear-equation2}
\psi_t=\psi_{xx}+a(\omega\cdot t,x)\psi_x+b(\omega\cdot t,x)\psi,\,\,t>0,\,x\in S^{1},
\end{equation}
where $\omega\in E$, $\omega\cdot t$ is a flow on a compact  metric space $E$; and $a^\omega(t,x):=a(\omega\cdot t,x)$,  $b^\omega(t,x):=b(\omega\cdot t,x)$ are continuously differentiable in $(t,x)$; and $a^\omega,a^\omega_t,a^\omega_x,b^\omega:\mathbb{R}\times S^1\to \mathbb{R}$ are bounded functions uniformly for $\omega\in E$.

For any $w\in L^2(S^1)$, let $\psi(t,x;w, \omega)$ be the solution  of \eqref{linear-equation2} with $\psi(0,x;w,\omega)=w(x),x\in S^1$.
\begin{definition}
{\rm An {\it exponentially bounded solution} $\psi(t,\cdot;w, \omega)$ of \eqref{linear-equation2} is a solution $\psi:\mathbb{R}\to L^2(S^1)$ defined for $t\in\mathbb{R}$  and there exist positive constants $K_1,K_2$ such that
\begin{equation}
    \|\psi(t,\cdot;\cdot, \cdot)\|_{L^2(S^1)}\leq K_1 e^{K_2|t|},\quad \forall t\in\mathbb{R}.
\end{equation}
}
\end{definition}

\begin{lemma}\label{zero-up-control}
 Assume $\psi(t,\cdot;w, \omega)$ be a nontrivial exponentially bounded solution of \eqref{linear-equation2}. Then, there exists some positive integer $N$ such that
\begin{equation*}
    z(\psi(t,\cdot;w, \omega))\leq N,\quad \forall t\in\mathbb{R}.
\end{equation*}
\end{lemma}
\begin{proof}
See \cite[Corrolary 4.5 and Section 9]{Chow1995}.
\end{proof}

\section{Asymptotic dynamics of semilinear heat equations with periodic boundary condition}

In this section, we study the longtime behavior of \eqref{equation-1} with periodic boundary condition \eqref{periodic-bc},
and prove the { Constancy Property Lemma} and Theorems \ref{center-constant}-\ref{structure-thm}

Assume that $X$ is the fractional power space as defined in the introduction. Given any $u\in X$ and $a\in S^1$, recall that the shift $\sigma_a$ on $u$ as $(\sigma_a u)(\cdot)=u(\cdot+a)$.
So, if $\varphi(t,\cdot;u,g)$ is a classical solution of \eqref{equation-lim1}, then it is easy to check that $\sigma_a\varphi(t,\cdot;u,g)$ is a classical solution of \eqref{equation-lim1}. Moreover, the uniqueness of solution ensures the {\it translation invariance}, that is, $\sigma_a\varphi(t,\cdot;u,g)=\varphi(t,\cdot;\sigma_au,g)$.

{A point $u\in X$ is called {\it spatially-homogeneous} if $u(\cdot)$ is independent of the spatial variable $x$. Otherwise, $u$ is called {\it spatially-inhomogeneous}. A subset $A\subset X$ is called {\it spatially-homogeneous} (resp. {\it spatially-inhomogeneous}) if any point in $A$ is spatially-homogeneous (resp. spatially-inhomogeneous).
Obviously, any minimal set $M$ of \eqref{equation-lim2} is either spatially-inhomogeneous or spatially-homogeneous.}

\subsection{Zero number constancy on the minimal set}

In this subsection, we prove the { Lemma \ref{center-constant}}. To do so, we first prove a lemma.

\begin{lemma}\label{asy-translate}
   Let $M$ be a spatially-inhomogeneous minimal set of \eqref{equation-lim2}. Then, for any $g\in H(f)$ and $(u_1,g), (u_2,g)\in M\cap p^{-1}(g)$, there is a sequence $t_n\to\infty$ (resp. $t_n\to-\infty$)such that $g\cdot t_n\to g^+$ (resp. $g\cdot t_n\to g^-$) and
\begin{equation}\label{asym-rota}
\varphi(t_n,\cdot;u_1,g)\to u_1^+,\,\,\, \varphi(t_n,\cdot;u_2,g)\to u_2^+\,\, (resp.\, \varphi(t_n,\cdot;u_1,g)\to u_1^-,\,\,\, \varphi(t_n,\cdot;u_2,g)\to u_2^- )
\end{equation}
with $u_1^+\in \Sigma u_2^+$ (resp. $u_1^-\in \Sigma u_2^-$).
\end{lemma}
\begin{proof}
Since $M$ is a spatially-inhomogeneous minimal set, it is easy to see that all the elements in $M$ share the same spatial-minimal period, that is, there exists $L^0>0$ such that for any $(u,g)\in M$, one has $u(\cdot+L^0)=u(\cdot)$ and $u(\cdot+a)\neq u(\cdot)$ for $a\in (0,L^0)$. Hereafter, we always assume $\mathcal{S}=\mathbb{R}/L^0\mathbb{Z}$.

In fact,  by taking a sequence $t_n\to \infty$, we assume that $g\cdot t_n\to g^+$ and
$\varphi(t_n,\cdot;u_1,g)\to u_1^+,\,\,\, \varphi(t_n,\cdot;u_2,g)\to u_2^+.$
If $u_1^+\in\Sigma u_2^+$, then the lemma holds.
If $u_1^+\not\in \Sigma u_2^+$, then by Lemma \ref{sequence-limit} and the connectivity of $\mathcal{S}$,
 there is an integer $\tilde N$ such that
$$
z(\varphi(t,\cdot;u_1^+,g^+)-\varphi(t,\cdot;\sigma_a u_2^+,g^+))=\tilde N,
\quad \text{for all }t\in \RR \text{ and }a\in \mathcal{S}.
$$
{By Lemma \ref{difference-lapnumber}, Lemma \ref{zero-cons-local} and} the compactness of $\mathcal{S}$, one can find a $T_0>0$ such that
$$
z(\varphi(t,\cdot;u_1,g)-\varphi(t,\cdot;\sigma_a u_2,g))=\tilde N,
\quad \text{for all }t\ge T_0 \text{ and }a\in \mathcal{S}.
$$
As a consequence,
$$
m_1(t):=\max _{x\in \mathcal{S}} \varphi(t,x;u_1,g)\not = m_2(t):= \max_{x\in \mathcal{S}}\varphi(t,x;u_2,g),\quad \text{for all }\,t\ge T_0.
$$
Without loss of generality, we may assume that
$m_1(t)>m_2(t)$ for all $t\ge T_0$. For the above $g^+$, let $m^+=\max\{\max_{x\in \mathcal{S}} u(x):(u,g^+)\in M\}$ and choose $u_2^{++}$ be such that
$(u_2^{++},g^+)\in M$ with $\max_{x\in \mathcal{S}}u_2^{++}(x)=m^+$. Since $M$ is minimal, one can take another sequence $t_n^+\to\infty$ such that
$$
(\varphi(t_n^+,\cdot;u_2,g),g\cdot t_n^+)\to (u_2^{++}(\cdot), g^+).
$$
For simplicity, we assume that $\varphi(t_n^+,\cdot;u_1,g)\to u_1^{++}(\cdot).$
Then, by the definition of $m^+$, we must have
\begin{equation}\label{max-equal}
\max_{x\in \mathcal{S}}u_1^{++}(x)=\max_{x\in \mathcal{S}}u_2^{++}(x).
\end{equation}
Suppose that $u_1^{++}\not \in  \Sigma_a u_2^{++}$. Then, again by Lemma \ref{sequence-limit} and the connectivity of $\mathcal{S}$,
there is $\tilde N^+$ such that
$$
z(\varphi(t,\cdot;u_1^{++},g^+)-\varphi(t,\cdot;\sigma_a u_2^{++},g^+))=\tilde N^+,\quad \forall\,\, t\in\RR,\,\, a\in \mathcal{S}.
$$
This contradicts to \eqref{max-equal}. Hence,
$u_1^{++}\in\Sigma u_2^{++}$
and the claim is proved. Similarly, there is $s_n\to -\infty$ such that $g\cdot s_n\to g^-$ and
\begin{equation}\label{asym-negative}
\varphi(s_n,\cdot;u_1,g)\to u_1^-,\,\,\, \varphi(s_n,\cdot;u_2,g)\to u_2^-,
\end{equation}
with $u_1^-\in \Sigma u_2^-$.
\end{proof}

We now prove Lemma \ref{center-constant}.

\begin{proof}[Proof of Lemma\ref{center-constant}]
Noticing that if $M$ is a spatially-homogeneous minimal set, then for any $g\in H(f)$, and two distinct points $(u_1,g),(u_2,g)\in p^{-1}(g)\cap M$, one has
\begin{equation*}
 z(\varphi(t,\cdot;u_1,g)-\varphi(t,\cdot;\sigma_au_2,g))=0,\quad \text{for any }\ t\in\mathbb{R},\ a\in S^1.
\end{equation*}
Therefore, we only consider the case that $M$ is spatially-inhomogeneous.

We prove that there exists $N\in \mathbb{N}$ such that for all $(u,g)\in M$, one has
\begin{equation}\label{contstant-S}
  z(\varphi(t,\cdot;u,g)-\varphi(t,\cdot;\sigma_au,g))=N, \quad \forall a\in \mathcal{S}\setminus\{0\}
\end{equation}
{where $\mathcal{S}$ be as defined in the proving of Lemma \ref{asy-translate}.}

Fix $(u_1,g_1)\in M$ and $a\in \mathcal{S}\setminus\{0\}$. Since $M$ is minimal, there exists $t_n\to \infty$ such that $\Pi_{t_n}(u_1,g_1)\to (u_1,g_1)$. By Lemma \ref{sequence-limit}, one has
\begin{equation}
  z(\varphi(t,\cdot;u_1,g_1)-\varphi(t,\cdot;\sigma_{a}u_1,g_1))=N_1
\end{equation}
for some $N_1$. Moreover, it is not hard to see that $N_1$ is independent of  the choice of $a\in \mathcal{S}\setminus\{0\}$. Similarly, for another given point $(u_2,g_2)\in M$, there is $N_2\in \mathbb{N}$ such that
\begin{equation}
  z(\varphi(t,\cdot;u_2,g_2)-\varphi(t,\cdot;\sigma_{a}u_2,g_2))=N_2, \quad \forall a\in \mathcal{S}\setminus\{0\}.
\end{equation}
 Again by the minimality of $M$, there exists $s_n\to\infty$ such that $\Pi_{s_n}(u_1,g_1)\to (u_2,g_2)$, by Lemma \ref{zero-cons-local}, it is also easy to see that $N_1=N_2$. For simplicity, we set $N=N_1=N_2$, \eqref{contstant-S} is then proved.

 In view of \eqref{contstant-S}, it is easy to see that \eqref{const-mini} always established for the given $N$ provided that $(u_1,g)$, $(u_2,g)\in M$ with $u_1\in \Sigma u_2$. Therefore, in the left of the proof, we always assume that $u_1\notin \Sigma u_2$ and will prove the following
\begin{equation}\label{contstant-S1}
  z(\varphi(t,\cdot;u_1,g)-\varphi(t,\cdot;\sigma_au_2,g))=N, \quad \forall t\in\mathbb{R},\, a\in \mathcal{S}.
\end{equation}

First of all, we show that
\begin{equation}\label{geq-S}
  z(\varphi(t,\cdot;u_1,g)-\varphi(t,\cdot;\sigma_au_2,g))\geq N,\quad \forall t\in\mathbb{R},\, a\in \mathcal{S}.
\end{equation}
By Lemma \ref{asy-translate}, one may assume that $\Pi_{t_n}(u_1,g)\to (u^+,g^+)$ and $\Pi_{t_n}(u_2,g)\to (\sigma_{a^+}u^+,g^+)$ for some $t_n\to\infty$, $(u^+,g^+)\in M$ and $a^+\in \mathcal{S}$. For any $a\in \mathcal{S}$ with $\sigma_{a+a^+}u^+\neq u^+$, by Lemma \ref{difference-lapnumber}(c), {Lemma \ref{zero-cons-local}} and equation \eqref{contstant-S}, there exists $T_a>0$ such that
\begin{equation}
  z(\varphi(t,\cdot;u_1,g)-\varphi(t,\cdot;\sigma_au_2,g))=z(u^+-\sigma_{a+a^+}u^+)=N, \quad t\geq T_a.
\end{equation}
Thus,
\begin{equation}\label{geq-inequality}
  z(\varphi(t,\cdot;u_1,g)-\varphi(t,\cdot;\sigma_au_2,g))\geq N, \quad \forall t\in\mathbb{R},\, \sigma_{a+a^+}u^+\neq u^+.
\end{equation}

For $a\in \mathcal{S}$ with $\sigma_{a+a^+}u^+= u^+$, again by Lemma \ref{difference-lapnumber}(c), there exist $T^+_{a}>0$ and $N_{a}^+\in \mathbb{N}\cup\{0\}$, such that
\begin{equation}\label{constant-1}
  z(\varphi(t,\cdot;u_1,g)-\varphi(t,\cdot;\sigma_{a}u_2,g))=N^+_{a}, \quad t\geq T^+_{a}.
\end{equation}
Particularly, we have
\begin{equation}
  z(\varphi(T^+_{a},\cdot;u_1,g)-\varphi(T^+_{a},\cdot;\sigma_{a}u_2,g))=N^+_{a}.
\end{equation}
Moreover, {by view of Lemma \ref{zero-cons-local}}, there is $\delta>0$ such that for all $|a-a_0|<\delta$, one has
\begin{equation}
  z(\varphi(T^+_{a},\cdot;u_1,g)-\varphi(T^+_{a},\cdot;\sigma_{a_0}u_2,g))=N^+_{a}.
\end{equation}
Since $M$ is spatially-inhomogeneous, there is $a^*$ with $|a^*-a|<\delta$ such that $\sigma_{a^*+a^+}u^+\neq u^+$ and
\begin{equation}\label{continuity}
  z(\varphi(T^+_{a},\cdot;u_1,g)-\varphi(T^+_{a},\cdot;\sigma_{a^*}u_2,g))=N^+_{a}.
\end{equation}
This together with \eqref{geq-inequality} implies $N_{a}^+\geq N$,  \eqref{geq-S} is thus proved.

To complete our proof, we still need to show
\begin{equation}\label{leq-S}
  z(\varphi(t,\cdot;u_1,g)-\varphi(t,\cdot;\sigma_au_2,g))\leq N, \quad \forall t\in\mathbb{R}, a\in \mathcal{S}.
\end{equation}
Following from Lemma \ref{asy-translate},  we assume that $\Pi_{s_n}(u_1,g)\to (u^-,g^-)$ and $\Pi_{s_n}(u_2,g)\to (\sigma_{a^-}u^-,g^-)$ for some $(u^-,g^-)\in M$ and $a^-\in \mathcal{S}$, while $s_n\to-\infty$. For $a\in \mathcal{S}$ with $\sigma_{a+a^-}u^-\neq u^-$, again by Lemma \ref{difference-lapnumber}(c), {Lemma \ref{zero-cons-local}} and  \eqref{contstant-S}, there exists $T_a>0$ such that
\begin{equation}
  z(\varphi(t,\cdot;u_1,g)-\varphi(t,\cdot;\sigma_au_2,g))=z(u^--\sigma_{a+a^-}u^-)=N, \quad t\leq -T_a.
\end{equation}
Thus, if $a\in \mathcal{S}$ satisfies  $\sigma_{a+a^-}u^-\neq u^-$, then

\begin{equation}
  z(\varphi(t,\cdot;u_1,g)-\varphi(t,\cdot;\sigma_au_2,g))\leq N, \quad \forall t\in\mathbb{R}.
\end{equation}

Suppose $a\in \mathcal{S}$ is that $\sigma_{a+a^-}u^-=u^-$, noticing that
$$0<\|\varphi(t,\cdot;u_1,g)-\varphi(t,\cdot;\sigma_au_2,g)\|<C,\ t\in \mathbb{R}$$
for some $C>0$ (since $X$ is compactly embed into $L^2(S^1)$, this means $\|\varphi(t,\cdot;u_1,g)-\varphi(t,\cdot;\sigma_au_2,g)\|_{L^2(S^1)}$ is exponential bounded). Thus, by Lemma \ref{zero-up-control}, there is an integer $N'_a\in\mathbb{N}\cup \{0\}$ such that
$$
z(\varphi(t,\cdot;u_1,g)-\varphi(t,\cdot;\sigma_au_2,g))\leq N'_a
$$
for all $t\in\mathbb{R}$. By Lemma \ref{difference-lapnumber}(c),
there exists $N^-_a\in\mathbb{N}$ and $T^-_a>0$ such that
\begin{equation}\label{negative-constant}
  z(\varphi(t,\cdot;u_1,g)-\varphi(t,\cdot;\sigma_au_2,g))=N^-_a, \quad \forall t\leq -T^-_a.
\end{equation}

Based on \eqref{negative-constant}, one can use similar arguments between \eqref{constant-1}-\eqref{continuity} to get
\eqref{leq-S}. Combining with \eqref{geq-S},
\begin{equation*}
  z(\varphi(t,\cdot;u_1,g)-\varphi(t,\cdot;\sigma_au_2,g))=N,\quad\forall t\in\mathbb{R},
\end{equation*}
for any $(u_1,g),(u_2,g)$ with $u_1\neq \sigma_a u_2$.

We have completed the proof of Lemma \ref{center-constant}.
\end{proof}

\subsection{Almost automorphically forced circle flows and almost periodically forced flow on $\mathbb{R}$}

In this subsection, we discuss the structure of compact minimal invariant set of \eqref{equation-lim2}, and prove
Theorems \ref{a-a-circle-flow} and \ref{a-p-flow}.

\begin{proof}[Proof of Theorem \ref{a-a-circle-flow}] We note that in the case that the dimension of the center space of $M$ is $2$ and the dimension of the unstable space of $M$ is odd, Theorem \ref{a-a-circle-flow} becomes \cite[Theorem 3.1 (1)]{SWZ}. In the general case, Theorem \ref{a-a-circle-flow} can
be proved by the same arguments as those in \cite[Theorem 3.1 (1)]{SWZ}.  For the clarity, we provide the outline of the proof in the following.
\smallskip

(1) As it is mentioned in the proof of Lemma  \ref{asy-translate},  it is easy to see that all the elements in $M$ share the same spatial-minimal period, that is, there exists $L^0>0$ such that for any $(u,g)\in M$, one has $u(\cdot+L^0)=u(\cdot)$ and $u(\cdot+a)\neq u(\cdot)$ for $a\in (0,L^0)$. In the following,  $\mathcal{S}=\mathbb{R}/L^0\mathbb{Z}$.

\smallskip

(2)
First, we introduce the quotient space $\tilde X$ of $X$ and the induced mapping of $\Pi_{t}$ on the quotient $\tilde X$  as follows.

For any $u,v\in X$, $u \sim v$ if and only if $u=\sigma_a v$ for some $a \in S^1$. It is easy to check that ``$\sim$'' is an equivalence relation on $X$,
denoted by $[u]$ for the equivalence class. Let $\tilde{X}=\{[u]|u\in X\}$, then $\tilde {X}$ is a quotient space of $X$. And hence, $\tilde{X}$ is a metric space with metric $\tilde{d}_{\widetilde{X}}$ defined as $\tilde{d}_{\tilde{X}}([u],[v]):=d_H(\Sigma u,\Sigma v)$ for any $[u],[v]\in \tilde{X}$
(Here $d_H(U,V)$ is the Hausdorff metric of the compact subsets $U,V$ in $X$, defined as $d_{H}(U,V)=\sup\{\sup_{u \in U} \inf_{v \in V} d_X(u,v),$ $\, \sup_{v \in V} \inf_{u \in U} d_X(u,v)\}$ and the metric $d_X(u,v)=||u-v||_{X}$). Note that $d_X$ satisfies the $S^1$-translation invariance:  $d_X(\sigma_a u,\sigma_a v)=d_X(u,v)$ for any $u,v\in X$ and any $a \in S^1$.
For any subset $K\subset X\times H(f)$, we write $\tilde{K}=\{([u],g)\in \tilde{X}\times H(f)|(u,g)\in K\}$.
Define the induced mapping $\tilde \Pi_{t}$ of  ${\Pi}_t$ ($t\geq 0$) on $\tilde{X}\times H(f)$ by
\begin{align}\label{induced-skepro-semiflow}
\begin{split}
\tilde{\Pi}_t:\tilde{X}\times H(f)&\rightarrow \tilde{X}\times H(f);\\
([u],g)&\rightarrow (\tilde{\varphi}(t,\cdot;[u],g),g\cdot t):=([\varphi(t,\cdot;u,g)],g\cdot t).
\end{split}
\end{align}

Second, note that the proof of Theorem \cite[Theorem 3.1 (1)]{SWZ} applies to any minimal set of $\Pi_t$ satisfies \cite[Corollary 3.9, Lemmas 3.10-3.12]{SWZ}. Hence Theorem  \ref{a-a-circle-flow} follows from the arguments of \cite[Theorem 3.1 (1)]{SWZ} provided that $M$ satisfies \cite[Corollary 3.9, Lemmas 3.10-3.12]{SWZ}.

Third, for $u\in X$, let $m(u)=\sup_{x\in S^1}u(x)$.  By  the constancy of the zero number on  the minimal set $M$ (see { Lemma} \ref{center-constant}),  $M$ satisfies \cite[Corollary 3.9]{SWZ}. That is, there is $N\ge 0$ such that for any $g\in H(f)$ and any two elements $(u_1,g),(u_2,g)$ in $M\cap p^{-1}(g)$,
{\it  \begin{itemize}
  \item[\rm{(i)}] $ z(\varphi(t,\cdot;u_1,g)-\varphi(t,\cdot;\sigma_au_2,g))=N$ for all $t\in\mathbb{R}$ and $a\in S^1$ with { $u_1\neq \sigma_au_2$};
  \item[\rm{(ii)}]$m(u_1)<m(u_2)$, then
$m(\varphi(t,\cdot;u_1,g))<m(\varphi(t,\cdot; u_2,g)),\  for\ all\ t\in \mathbb{R}$;
\item[\rm{(iii)}]
$m(u_1)=m(u_2)$ $\Leftrightarrow$ $([u_1],g)=([u_2],g)$.
  \end{itemize}
}

Fourth, it can be directly verified that \cite[Lemma 3.10]{SWZ} holds. That is, we have
{\it \begin{itemize}
\item [\rm{(iv)}] $\tilde{\Pi}_t$ admits a skew-product semiflow on $\tilde{X}\times H(f)$;
\item [\rm{(v)}] If $M$ is a minimal subset in $X\times H(f)$ w.r.t. $\Pi_{t}$, then $\tilde{M}$ is also a minimal subset in $\tilde{X}\times H(f)$ w.r.t. $\tilde{\Pi}_t$.
\end{itemize}
}

Fifth,  let $\tilde{p}:\tilde{X}\times H(f)\rightarrow H(f)$ be the natural projection. Define an ordering on each fiber $\tilde{M}\cap \tilde{p}^{-1}(g)$, with the base point $g\in H(f)$ as follows:
$$([u],g)\leq_g ([v],g)\,\, \text{ if }\,\, m(u)\leq m(v).$$
We also write the strict relation $([u],g)<_g ([v],g)$ if $m(u)<m(v)$. Without any confusion, we hereafter will drop the subscript $``g"$.
By  the constancy of the zero number on  the minimal set $M$ (see { Lemma} \ref{center-constant}) again, it can be proved that $M$ satisfied \cite[Lemma 3.11]{SWZ}. That is,
{\it \begin{itemize}
\item[(vi)]
$``\leq"$ is a total ordering on each $\tilde{M}\cap \tilde{p}^{-1}(g)$, ($g\in H(f)$) and $\tilde{\Pi}_t$ is strictly order preserving on $\tilde{M}$ in the sense that, for any $g\in H(f)$, $([u],g)<([v],g)$ implies that $\tilde{\Pi}_t([u],g)<\tilde{\Pi}_t([v],g)$ for all $t\geq 0$.
\end{itemize}
}

Sixth,  let $E\subset \tilde{X}\times H(f)$ be a compact invariant subset of $\tilde{\Pi}_t$ which admits a flow extension.
For each $g\in H(f)$, we define {\it a fiberwise strong ordering $``\ll"$ on each fiber} $E\cap\tilde{p}^{-1}(g)$ as follows: $([u_1],g)\ll ([u_2],g)$ if there exist neighborhoods $\mathcal{N}_1,\mathcal{N}_2 \subset  E\cap\tilde{p}^{-1}(g)$ of $([u_1],g),([u_2],g)$, respectively, such that $([u^*_1],g)< ([u^*_2],g)$ for all $([u_i^{*}],g)\in \mathcal{N}_i\ (i=1,2).$
 Moreover, for each $g\in H(f)$, we say $([u_1],g),([u_2],g)$ forms {\it a strongly order-preserving pair} if $([u_1],g),([u_2],g)$ is strongly ordered on the fiber, written $([u_1],g)\ll([u_2],g)$, and there are neighborhoods $U_i$ of $([u_i],g)$ $(i=1,2)$ in $E$ respectively, such that whenever $([u^*_1],g),([u^*_2],g)\in E\cap \tilde{p}^{-1}(g)$, with $\tilde{\Pi}_{T}([u^*_1],g)\in U_i\ (i=1,2)$ for some $T<0$, then $([u^*_1],g)\ll([u^*_2],g)$. We have \cite[Lemma 3.12]{SWZ} is also satisfied. That is,
{\it \begin{itemize}
\item[(vii)]Let $\tilde M$ be a minimal set of $\tilde{\Pi}_t$ which admits a flow extension and $Y'$ be as in Lemma \ref{epimorphism-thm}. Then for any $g\in Y',\, \tilde M\cap\tilde{p}^{-1}(g)$ admits no strongly order preserving pair.
\end{itemize}
}


Let $Y_0=Y'$, where $Y'$ is as defined in Lemma \ref{epimorphism-thm}. By virtue of (iv)-(v), we consider the induced minimal set $\tilde{M}$ for the skew-product semiflow $\tilde{\Pi}_t$ on $\tilde{X}\times H(f)$.

We show that for any $g\in Y_0$, there exists $u_g\in X$ such that $M\cap p^{-1}(g)\subset (\Sigma u_g,g)$, which equivalents to prove that $\tilde{M}\cap \tilde{p}^{-1}(g)$ is a singleton. Suppose on the contrary that there are two distinct points $([u_1],g),([u_2],g)$ on $\tilde{M}\cap \tilde{p}^{-1}(g)$ for some $g\in Y_0$. Then by (vi) and the same argument as those in \cite[Theorem 3.1 (1)]{SWZ}, one can get $([u_1],g)$ and $([u_2],g)$ forms a strongly order preserving pair.

On the other hand, (vii) implies that there exists no such strongly order preserving pair on $\tilde{M}\cap \tilde{p}^{-1}(g)$, {which forms} a contradiction.  Thus, for any $g\in Y_0$, $\tilde{M}\cap \tilde{p}^{-1}(g)$ is a singleton (In other words, $\tilde{M}$ is an almost $1$-cover of $H(f)$), which also implies that $M\cap p^{-1}(g)\subset (\Sigma u_g,g)$ for some $u_g\in X$. (2) is thus proved.

\smallskip

(3)  
Define the mapping
\begin{equation}\label{E:natu-proj-h}
h:\tilde{M}\to \mathbb{R}\times H(f); ([u],g)\mapsto (m(u),g).
\end{equation}
Let $\hat{M}=h(\tilde{M})$, one can naturally define the skew-product flow on $\hat{M}\subset \mathbb{R}\times H(f)$:
\begin{equation}\label{E:induced-flow-hat-M}
\hat{\Pi}_t:\hat{M}\to \hat{M};(m(u),g)\mapsto (m(\varphi(t,\cdot,u,g)),g\cdot t),
\end{equation}
which is induced by $\Pi_{t}$ restricted to $M$. By (iii) and the arguments as that in \cite[Theorem 3.1 (1)]{SWZ}, $h$ is a topologically-conjugate homeomorphism between $\tilde{M}$ and $\hat{M}\subset \mathbb{R}\times H(f)$. As a consequence, $\hat{M}$ is also an almost $1$-cover of $H(f)$ (with the residual subset $Y_0\subset H(f)$).

For each $g\in Y_0$, we choose some element, still denoted by $u_g(\cdot)$, from the $S^1$-group orbit $\Sigma u_g$ such that
\begin{equation}\label{E:u-g--base-trn}
\quad u_g(0)=m(u_g), \quad\hat M\cap p^{-1}(g)=(m(u_g),g)\,\text{ and }\,\tilde{M}\cap \tilde{p}^{-1}(g)=([u_g],g).
\end{equation}
Together with  \eqref{E:induced-flow-hat-M}, \eqref{E:u-g--base-trn} implies that $$u_{g\cdot t}(0)=m(u_{g\cdot t})=m(\varphi(t,\cdot,u_g,g)),\,\,\,\text{ for any }g\in Y_0\,\,\text{ and }t\in \mathbb{R}.$$
The above function
$t\mapsto u_{g\cdot t}(0)$ is clearly continuous and is almost automorphic in $t$, and $u^{ g}(t,x):=u_{g\cdot t}(x)$ is almost automorphic in $t$ uniformly in $x$. Moreover,
\begin{equation}\label{E:u'-u''-nontrivial-1}
u^{'}_g(0)=0\,\text{ and }u^{''}_g(0)\not =0\quad \text{ for any }g\in Y_0
\end{equation}
due to the minimality and spatial-inhomogeneity of $M$.

For given $g\in Y_0$, let
$$
G(t;g)=g_p(t,u_{g\cdot t}(0),0)+\frac{u^{'''}_{g\cdot t}(0)}{u^{''}_{g\cdot t}(0))}.
$$
It is clear that $G(t;g)$ is almost automorphic.
Let
$$
M_0=\{(u,g)\in M\,|\, g\in Y_0\}
$$
and
$$
\bar M_0=\{(c(u),G(\cdot;g))\,|\, (u,g)\in M_0,\,\, c(u)\in \mathcal{S} \,\, \text{is such that}\,\, u(\cdot)=u_g(\cdot+c(u))\},
$$
It is also clear that
$$\bar h_0:M_0\to \bar M_0,\,\, \bar h_0(u,g)=(c(u),G(\cdot;g))
$$
is a homeomorphism
\smallskip

(4) For any $g\in Y_0$, we can define a function $t\mapsto c^g(t)\in \mathcal{S}$ such that
\begin{equation}\label{E:rotation-spiral}
 \varphi(t,x;u_g,g)= u_{g\cdot t}(x+ c^g(t)),\quad \text{or equivalently,}\quad \varphi(t,x-c^g(t);u_g,g)=u_{g\cdot t}(x).
\end{equation}
Similarly as the arguments in \cite[Theorem 3.1 (1)]{SWZ}, the function $t\mapsto c^g(t)\in \mathcal{S}$ is continuous.

By \eqref{E:rotation-spiral} and the property of $u_{g\cdot t}(x)$ in \eqref{E:u'-u''-nontrivial-1}, we observe that
$$
\varphi_x(t,-c^g(t);u_g,g)=u^{'}_{g\cdot t}(0)=0\,\,
\text{ and }\,\,\varphi_{xx}(t,-c^g(t);u_g,g)= u^{''}_{g\cdot t}(0)\not =0.
$$
Then by the continuity of $c^g(t)$ in $t$ and Implicit Function Theorem, we have $c^g(t)$ is differentiable in $t$; and moreover
\begin{eqnarray}\label{circle-flow-eq41}
\dot {c}^g(t)=G(t;g)
 \end{eqnarray}
and
 hence $\dot{c}^g(t)$, is time almost-automorphic in $t$.   Let
 $$
{\bar \Pi_t^0: \bar M_0\to\bar M_0,\quad \bar\Pi_t^0 (c,G)=(c+c^g(t),G\cdot t)}
$$
It is not difficult to see that $\bar \Pi_t$ is a flow on $\bar M_0$, and
$$
\bar h_0\Pi_t (u,g)=\bar\Pi_t^0 \bar h_0(u,g)\quad \forall \, { t\in\mathbb{R}},\, \, (u,g)\in M_0.
$$
Since $\bar h_0$ is a homeomorphism, to prove it is a topological-conjugate between two flows we only need to check the above equality is correct, indeed
\begin{equation*}
\begin{split}
\bar h_0\Pi_t (u,g)&=(c(\varphi(t,\cdot;u_g(\cdot+c(u)),g)),G\cdot t) \\
&=(c(\varphi(t,x+c(u);u_g,g)),G\cdot t)\\
&=(c(u_{g\cdot t}(x+c(u)+c^g(t))),G\cdot t)\\
&=(c(u)+c^g(t),G\cdot t)=\bar\Pi_t^0 \bar h_0(u,g).
\end{split}
\end{equation*}
This completes the proof of (4).
\end{proof}

To prove Theorem \ref{a-p-flow}, we need the following important lemma from \cite[Proposition 3.2]{SWZ2}.

\begin{lemma}\label{order}
Assume that $f(t,u,u_x)=f(t,u,-u_x)$ in \eqref{equation-1}. Let $(u_0,g_0)\in X\times H(f)$ be such that the motion $\Pi_{t}(u_0,g_0)$($t>0$) of \eqref{equation-lim2} {is bounded} and $\omega(u_0,g_0)$ be the $\omega$-limit set. Then, there is a point $x_0\in S^1$ such that for any $(u,g)\in \omega(u_0,g_0)$, one has $u_x(x_0)=0$.
\end{lemma}
\begin{proof}[Proof of Theorem \ref{a-p-flow}]
By { Lemma} \ref{center-constant}, one has
 \[
  z(\varphi(t,\cdot;u_1,g)-\varphi(t,\cdot;u_2,g))=N,
 \]
for any $t\in\mathbb{R}$, $(u_1,g), (u_2,g)\in M$. Therefore, $\varphi(t,\cdot;u_1,g)-\varphi(t,\cdot;u_2,g)$ has only simple zeros on $S^1$ for all $t\in\mathbb{R}$. Particularly, $u_1-u_2$ has only simple zeros on $S^1$. By Lemma \ref{order}, $u_1(x_0)-u_2(x_0)\neq 0$, where $x_0$ as defined in Lemma \ref{order}.

For such $x_0$,  we define the following mapping:
\begin{equation}\label{embeding-map}
\bar h:M\longrightarrow \mathbb{R}\times H(f);(u,g)\longmapsto (u(x_0),g).
\end{equation}
Clearly, $h$ is continuous injective and onto $\bar M=\bar h(M)\subset \mathbb{R}\times H(f)$.
Thus, $\Pi_{t}$ naturally induces a (skew-product) flow $\bar{\Pi}_t$ on $\bar{M}$ as:
\begin{equation}\label{topological-conju}
\bar{\Pi}_t(h(u,g))\triangleq \bar h(\varphi(t,\cdot;u,g)(x_0),g\cdot t)\,\,\,\text{ for any }\bar h(u,g)\in \bar{M}.
\end{equation}
We show that the map $(\bar h|_{_{M}})^{-1}$ is also continuous from $\bar{M}$ to $M$. Indeed, let $\bar h(u^n,g^n)\to \bar h(u,g)$ in $\bar{M}$ (that is, $(u^n(x_0),g^n)\to (u(x_0),g)$ with $g^n\to g$ in $H(f)$). By the compactness of $M$,
one may assume without loss of generality that $(u^n,g^n)\to (w,g)\in M$. This then implies that
$u(x_0)=v(x_0)$. Recall that $(u,g),(v,g)\in M$ with $g\in H(f)$. Suppose that $u\ne v$. Then { Lemma} \ref{center-constant} implies that $u-v$ possesses only simple zeros, a contradiction. Consequently, $u=v$, and hence,  $(u^n,g^n)\to (u,g)\in M$. Thus, we have proved  $(h|_{_{M}})^{-1}$ is continuous from $\bar{M}$ to $M$. By virtue of \eqref{topological-conju}, $(M,\Pi_{t})$ is topologically conjugate to the flow $(\bar{M},\bar{\Pi}_t)$ on $\mathbb{R}\times H(f)$.
\end{proof}

\subsection{Structure of $\Omega$-limit sets}

In this section, we study the structure of $\omega$-limit set of \eqref{equation-lim2} and prove Theorem \ref{structure-thm}.

\begin{proof} [Proof of Theorem \ref{structure-thm}]
It can be proved by the similar arguments as those in
  \cite[Theorem 5.3]{SWZ3}, {while its proof involves \cite[Lemma 3.7, Lemma 4.1, Corollary 4.2, Lemma 4.3, Lemma 5.4]{SWZ3}. For the clarity, we will check that all the lemmas and corollary used for showing \cite[Theorem 5.3]{SWZ3} are also valid in the current situation.}

First, in view of the result in Theorem \ref{a-a-circle-flow}, by using similar deductions as those in \cite[Lemma 3.7]{SWZ3} one can immediately get
{\it
\begin{itemize}
\item[\rm{(i)}]
Let $M_1, M_2\subset \Omega$ be two minimal sets with $\Sigma M_1\cap M_2=\emptyset$. Then, there exists an integer $N\in \mathbb{N}$ such that \begin{equation}\label{constant-0}
z(\varphi(t,\cdot;\sigma_{a_1}u_1,g)-\varphi(t,\cdot;\sigma_{a_2}u_2,g))=N,
\end{equation}
for any $t\in \mathbb{R}$, $g\in H(f)$, $(u_i,g)\in M_i$ and $a_i\in S^1$, $i=1,2$.
\end{itemize}
}

Second,
hereafter we always write
\begin{equation}\label{E:Induced-Omega}
\tilde \Omega=\{([u],g)|(u,g)\in \Omega\}.
\end{equation}
Then, following the arguments of \cite[Lemma 4.1]{SWZ3}, we have
{\it\begin{itemize}
\item[\rm{(ii)}] $\tilde \Omega=\omega([u_0],g_0)$, where
$$\omega([u_0],g_0)=\{([u],g)\, |\, \text{ there exists }t_n\to \infty \text{ such that }\tilde\Pi_{t_n}([u_0],g_0)\to ([u],g)\}.$$
Particularly, if $\tilde \Omega=\tilde M$ is a minimal set of $\tilde\Pi_{t}$, then there is a minimal set $M\subset X\times H(f)$ ($M\subset \Omega$) such that $\tilde M=\{([u],g)|(u,g)\in M\}$.
\end{itemize}
}

Third,
let $\tilde M_1$, $\tilde M_2\subset \tilde \O$ be two minimal sets of $\tilde\Pi_{t}$ and $M_1,M_2\subset \Omega$ be two minimal sets of $\Pi_{t}$ such that $\tilde M_i=\{([u_i],g)|(u_i,g)\in M_i\}$ ($i=1,2$). Define
\begin{eqnarray}\label{E:min-max}
 \begin{split}
   m_i(g):=\min\{m(u_i)|(u_i,g)\in M_i\cap p^{-1}(g)\},\\
   M_i(g):=\max\{m(u_i)|(u_i,g)\in M_i\cap p^{-1}(g)\}
 \end{split}
 \end{eqnarray}
 for $i=1,2$.  By the arguments of  \cite[Lemma 4.3 (ii)]{SWZ3}, we have
 {\it
 \begin{itemize}
   \item[{ \rm (iii)}] $[m_1(g),M_1(g)]\cap[m_2(g),M_2(g)]=\emptyset$ for all $g\in H(f)$;
   \item[{ \rm (iv)}] If $m_2(\tilde g)>M_1(\tilde g)$ for some $\tilde g\in H(f)$, then there exists $\delta>0$ such that  $m_2(g)>M_1(g)+\delta$ for all $g\in H(f)$.
\end{itemize}
}

 Fourth, following the arguments of  \cite[Lemma 5.4]{SWZ3}, we have that
{\it

\begin{itemize}
\item[\rm{(v)}] $\tilde\Omega$ contains at most two minimal sets of $\tilde \Pi_{t}$; and moreover, one of the following  three alternatives must occur:
\begin{itemize}
\item[{ \rm (v-a)}] $\tilde \O$ is a minimal set of $\tilde \Pi_{t}$;
\item[{ \rm (v-b)}] $\tilde \O=\tilde M_1\cup \tilde M_{11}$, where $\tilde M_1$ is minimal, $\tilde M_{11}\neq \emptyset$, $\tilde M_{11}$ connects $\tilde M_1$ in the sense that if $([u_{11}],g)\in \tilde M_{11}$, then $\omega([u_{11}],g)\cap \tilde M_1\neq \emptyset$, and $\alpha ([u_{11}],g)\cap \tilde M_1\neq \emptyset$;
\item[{ \rm (v-c)}] $\tilde \O=\tilde M_1\cup \tilde M_2\cup \tilde M_{12}$, where $\tilde M_1$, $\tilde M_2$ are minimal sets, $\tilde M_{12}\neq \emptyset$ and connects $\tilde M_1$, $\tilde M_2$ in the sense that if $([u_{12}],g)\in \tilde M_{12}$, then $\omega([u_{12}],g)\cap(\tilde M_1\cup \tilde M_2)\neq\emptyset$ and $\alpha([u_{12}],g)\cap(\tilde M_1\cup \tilde M_2)\neq\emptyset$.
\end{itemize}
\end{itemize}
}

Thus, to get Theorem \ref{structure-thm} we use the deductions as in \cite[Theorem 5.3]{SWZ3} in the following.

{Observe that $\tilde \O=\{([u],g)|(u,g)\in \Omega\}$, in the case (v-a), one has $\tilde \O=\tilde M$; and hence, (ii) implies that there is a minimal set $M\subset \O$ such that $\tilde M=\{([u],g)|(u,g)\in M\}$. Suppose that there is a point $(u_*,g)\in \O$, but $(u_*,g)\notin \Sigma M$. Then $u_*\neq \sigma_a u$ for any $a\in S^1$ and $(u,g)\in M$, which means that $([u_*],g)\notin \tilde M$, a contradiction to $([u_*],g)\in \tilde \O=\tilde M$. Thus, $\O\subset \Sigma M$.

When (v-b) holds, that is, $\tilde \O=\tilde M_1\cup \tilde M_{11}$, where $\tilde M_1$ is a minimal set of $\tilde \Pi_{t}$ and $\tilde M_{11}\neq \emptyset$. Again by using (ii), one can choose a minimal set $M_1\subset \O$ such that $\tilde M_1=\{([u],g)|(u,g)\in M_1\}$. Let $M_{11}=\O\setminus \Sigma M_1$. Then it is easy to see that $\tilde M_{11}=\{([u],g)|(u,g)\in M_{11}\}$; and moreover, there is no minimal set in $M_{11}$.
Hence, one can assert that both $\Sigma M_1\cap \omega(u_{11},g)\ne\emptyset$ and $\Sigma M_1\cap \alpha(u_{11},g)\ne\emptyset$. In fact, suppose on the contrary that $\Sigma M_1\cap \omega(u_{11},g)=\emptyset$. Then one can find a minimal set $M_2\subset \omega(u_{11},g)$. Therefore, $M_2\cap \Sigma M_1=\emptyset$, that is, $\Sigma M_2\cap \Sigma M_1=\emptyset$. Let $\tilde M_2=\{([u],g)|(u,g)\in M_2\}$, then $\tilde M_2\neq \tilde M_1$ is also a minimal set of $\tilde\Pi_{t}$ contained in $\tilde \Omega$, a contradiction. Thus, we have proved (2). Similarly, we can also prove (3) as long as (v-c) holds.}
\end{proof}

\section{Asymptotic dynamics of semilinear heat equations with Neumann/Dirichelt boundary condition}

In this section, we study the asymptotic dynamics of \eqref{equation-1} with Neumann boundary condition \eqref{neumann-bc}
or Dirichlet boundary condition \eqref{dirichlet-bc}, and prove Theorems \ref{Ne-imbed} and \ref{Di-imbed}.

\subsection{Neumann boundary condition}
\label{Ne-convert}

In this subsection, we consider the minimal set generated by the skew-product semiflow of \eqref{Ne-skew-product} and prove Theorem \ref{Ne-imbed}.

 \begin{proof} [Proof of Theorem \ref{Ne-imbed}]
 Let $M\subset X_N\times H(f)$ be a minimal set of $\Pi_t^N$. For any $(u_0,g)\in M$, by the regularity of parabolic equations,
 $u_0\in C^2[0,L]$.  Let
 \begin{equation*}
 \tilde u_0(x)=\begin{cases}
 u_0(x),\quad x\in [0,L]\cr
 u_0(-x),\quad x\in [-L,0].
 \end{cases}
\end{equation*}
Then $\tilde u_0\in C^2([-L,L])$, $\tilde u_0(-L)=\tilde u_0(L)$,
$\tilde u_0^{'}(-L)=\tilde u_0^{'}(L)=0$, and $\tilde u_0(x)=\tilde u_0(-x)$ for $x\in [-L,L]$.
Let $\tilde u(t,x)=\tilde\phi(t,x;\tilde u_0,g)$ be the solution of
\begin{equation}
\label{neumann-extended-eq-bc}
\begin{cases}
u_t=u_{xx}+g(t,u,u_x),\quad -L<x<L\cr
u(t,-L)=u(t,L),\,\, u_x(t,-L)=u_x(t,L)
\end{cases}
\end{equation}
with $\tilde u(0,x)=\tilde u_0(x)$. By {\bf (HNB)}, $\tilde u(t,x)=\tilde u(t,-x)$ and then
$\tilde u_x(t,-L)=\tilde u_x(t,L)=\tilde u_x(t,0)=0$. This implies that
$$
\phi^N(t,x;u_0,g)=\tilde\phi(t,x;\tilde u_0,g)\quad \forall\, t\in\RR,\,\, x\in [0,L].
$$

Let
$$
\tilde M=\{(\tilde u_0,g)\,|\, (u_0,g)\in M\}.
$$
$\tilde M$ is a minimal set of the skew-product semiflow generated by \eqref{neumann-extended-eq-bc}.
Recall that
$$
M_0^N=\{(u_0(0),g)\,|\, (u_0,g)\in M\}.
$$
By { Lemma} \ref{center-constant}, the mapping $M\ni(u_0,g)\to (u_0(0),g)\in M_0^N$ is a continuous bijection,
and $\Pi_t^N|_M$ is conjugate to $\pi_t^N:M_0^N\to M_0^N$, where
$$
\pi_t^N(u_0(0),g)=(\phi^N(t,0;u_0,g),g\cdot t)\quad \forall\, (u_0,g)\in M.
$$
Theorem \ref{Ne-imbed} is thus proved.
 \end{proof}

\subsection{Dirichlet boundary condition}
\label{Di-convert}

In this subsection, we consider the minimal set generated by the skew-product semiflow of \eqref{Di-skew-product} and prove Theorem \ref{Di-imbed}.

\begin{proof}[Proof of Theorem \ref{Di-imbed}]
First, we assume that {\bf (HDB1)} holds.  Let $M\subset X_D\times H(f)$ be a minimal set of $\Pi_t^D$. For any $(u_0,g)\in M$,
let
$$
\tilde u_0(x)=\begin{cases}
u_0(x),\quad x\in [0,L]\cr
-u_0(-x),\quad x\in [-L,0].
\end{cases}
$$
By the regularity of parabolic equations,
$u(t,x)=\phi^D(t,x;u_0,g)$ is $C^1$ in $t\in\RR$ and $C^2$ in $x\in [0,L]$. Note that
$u(t,0)=u(t,L)=0$ for $t\in\RR$. Hence $u_t(t,0)=u_t(t,L)=0$.
This together with {\bf (HDB1)} implies that
\begin{align*}
0=u_t(t,0)&=u_{xx}(t,0)+f(t,u(t,0),u_x(t,0))\\
&=-u_{xx}(t,0)-f(t,u(t,0),u_x(t,0))\\
&=-u_{xx}(t,0)+f(t,-u(t,0),u_x(t,0)).
\end{align*}
Hence $u_{xx}(t,0)=-u_{xx}(t,0)=0$. Similarly, $u_{xx}(t,L)=0$.
In particular,  we have
$$
\tilde u_0^{''}(0-)=\tilde u_0^{''}(0+)=0,\quad \tilde u_0^{''}(-L+)=\tilde u_0^{''}(L)=0.
$$
Therefore, $\tilde u_0\in C^2[-L,L]$ and
$$
\tilde u_0(-L)=\tilde u_0(0)=\tilde u_0(L)=0,\quad \tilde u_0^{'}(-L)=\tilde u_0^{'}(L).
$$

Let $\tilde u(t,x):=\tilde\phi(t,x;\tilde u_0,g)$ be the solution of
\begin{equation}
\label{dirichlet-extended-eq-bc}
\begin{cases}
u_t=u_{xx}+g(t,u,u_x),\quad -L<x<L\cr
u(t,-L)=u(t,L),\,\, u_x(t,-L)=u_x(t,L)
\end{cases}
\end{equation}
with $\tilde u(0,x)=\tilde u_0(x)$.
By {\bf (HDB1)}, $\tilde u(t,x)=-\tilde u(t,-x)$. Hence $\tilde u(t,-L)=\tilde u(t,L)=-\tilde u(t,-L)=0$
and $\tilde u(t,0)=0$. This implies that
$$
u(t,x)=\tilde u(t,x),\quad x\in [0,L].
$$

Let
$$
\tilde M=\{(\tilde u_0,g)\,|\, (u_0,g)\in M\}.
$$
Then $\tilde M$ is a minimal set of the skew-product semiflow generated by \eqref{dirichlet-extended-eq-bc}.
Recall that
$$
M_0^D=\{(u_0^{'}(0),g)\,|\, (u_0,g)\in M\}.
$$
By { Lemma} \ref{center-constant}, the mapping $M\ni(u_0,g)\to (u_0(0),g)\in M_0^D$ is a continuous bijection,
and $\Pi_t^D|_M$ is conjugate to $\pi_t^D:M_0^D\to M_0^D$, where
$$
\pi_t^D(u_0(0),g)=(\phi_x(t,0;u_0,g),g\cdot t),\quad \forall\, (u_0,g)\in M.
$$
Theorem \ref{Di-imbed} is thus proved when {\bf (HDB1)} holds.

Next, we assume that {\bf (HDB2)} holds and $u_0\ge 0$ for any $(u_0,g)\in M$.
Let
$$
\tilde f(t,u,p)=\begin{cases}
f(t,u,p),\quad u\ge 0\cr
-f(t,-u,p),\quad u<0.
\end{cases}
$$
Then $\tilde f$ is $C^1$. Since $u_0\ge 0$ for any $(u_0,g)\in M$, we can replace $f$ by $\tilde f$ and get
$M\subset X_D\times H(\tilde f)$. It then follows from the above arguments that
the mapping $M\ni(u_0,g)\to (u_0(0),g)\in M_0^D$ is a continuous bijection,
and $\Pi_t^D|_M$ is conjugate to $\pi_t^D:M_0^D\to M_0^D$, where
$$
\pi_t^D(u_0(0),g)=(\phi_x(t,0;u_0,g),g\cdot t)\quad \forall\, (u_0,g)\in M.
$$
This completes the proof of Theorem \ref{Di-imbed}.
\end{proof}

\section{An example for quasi-periodic case}
In this section, we give an example that to show that even for quasi-periodic case, quasi-periodically-forced circle flow may not always induced by equation \eqref{equation-1}+\eqref{periodic-bc}(the example is from Fink \cite[Example 12.5]{Fink} and also used in \cite{ShenYi-TAMS95,Shen1998}). Consider a differential equation on the torus
\begin{equation}\label{ODE-example}
\dot{x} = f (t,x)
\end{equation}
where $f(t + 1, x ) = f (t,x + 1 ) = f(t, x)$. Let $x ( t , \eta )$ be the solution of \eqref{ODE-example} with $x ( 0 , \eta ) = \eta .$ Define the Poincar\'{e} map $\psi: \eta \mapsto x ( 1 , \eta )$. When the rotation number $\rho$ of $\psi$ is irrational, then the $\omega$-limit set $\omega_{\psi}(\eta)$ of
$\left\{ \psi ^ { n } ( \eta ) \text { mod } 1 , n = 1,2 , \ldots \right\}$ is either $[0,1]$ or is a cantor set (see \cite{Fink}).

Now, let $f$ in \eqref{ODE-example} be such that $\omega_{\psi}(\eta)$ is a cantor set. The following equation
\begin{equation}\label{induced-quasi}
\dot{x} = f (t, x + \rho t) - \rho
\end{equation}
is quasi-periodic dependent on $t$ with frequencies 1 and $\rho .$ It
is shown by Fink \cite{Fink} that equation \eqref{induced-quasi} admits a bounded
solution but no almost periodic solution, which means the skew-product flow generated by \eqref{induced-quasi} admits a non-almost periodic almost automorphic minimal set (see \cite[Example 3.2]{Shen1998}). Therefore, any minimal set of \eqref{induced-quasi} cannot even be embedded into a minimal set of some almost periodically-forced circle flow.

\section*{Acknowlegement}
Dun Zhou would like to thank the Chinese Scholarship Council (201906845011) for its financial support during his overseas visit and express his gratitude to the Department of Mathematics and Statistics of Auburn University for  its hospitality.


\begin{thebibliography}{10}
\bibitem{Amann}
H. Amann, \rm{Existence and regularity for semilinear parabolic evolution equations}, Ann. Scuola Norm. Sup. Pisa Cl. Sci. (4) \textbf{11} (1984), no. 4, 593-676.
\bibitem{2038390}
S. Angenent, \rm{The zero set of a solution of a parabolic equation}, J. Reine
 Angew. Math. \textbf{390} (1988), 79--96.

\bibitem{Angenent1988}
S. Angenent and B. Fiedler, \rm{The dynamics of rotating waves in scalar reaction diffusion equations}, Trans. Amer. Math. Soc. \textbf{307} (1988), 545--568.
\bibitem{P.Bates}
P. W. Bates and C. K. Jones, \rm{Invariant manifolds for semilinear partial differential equations},  Dynamics reported, Vol. 2 (1989), 1-38.

\bibitem{BPS}
P. Brunovsky, P. Pol\'{a}\v{c}ik and B. Sandstede, \rm{Convergence in general periodic parabolic equations in one space dimension}, Nonlinear Analysis, Theory, Method and Applicationions. \textbf{18} (1992),
no. 3, 209-215.
\bibitem{CCH}
M. Chen, X. Chen and J.K. Hale, \rm{Structural stability for time-periodic one-dimensional parabolic equations},J. Differential Equations \textbf{96} (1992), no. 2, 355-418.
\bibitem{Chen98}
X. Chen, \rm{A strong unique continuation theorem for parabolic equations}, Math. Ann. \textbf{311} (1998), 603--630.

\bibitem{Chen1989160}
X. Chen and
 H.~Matano, \rm{Convergence, asymptotic periodicity, and
  finite-point blow-up in one-dimensional semilinear heat equations}, J.
  Diff. Eqns. \textbf{78} (1989), 160-190.

\bibitem{Chen-P}
X. Chen and P. Pol\'{a}\v{c}ik, \rm{Gradient-like structure and Morse decompositions for time-periodic one-dimensional parabolic equations}, J. Dynam. Differential Equations \textbf{7} (1995), no. 1, 73-107.
\bibitem{Chow1994}
S.-N. Chow, and H. Leiva,\rm{Dynamical spectrum for time dependent linear systems in Banach spaces}, Japan J. Indust. Appl. Math., \textbf{11} (1994),379-415.

\bibitem{CLM1994}
S.-N. Chow, K. Lu, and J.~Mallet-Paret, \rm{Floquet Theory for Parabolic Differential Equations}, J. Diff. Eqns. \textbf{109} (1994), 147-200.
\bibitem{Chow1995}
S.-N. Chow, K. Lu, and J.~Mallet-Paret, \rm{Floquet bundles for scalar parabolic equations}, Arch.
  Ration. Mech. Anal., \textbf{129} (1995), 245-304.

\bibitem{Chow1991}
S.-N. Chow, X. Lin, and K. Lu, \rm{Smooth invariant foliations in infinite-dimensional spaces}, J. Diff. Eqns. \textbf{94} (1991), 266-291.

\bibitem{Chow1994-2}
S.-N. Chow and Y. Yi, \rm{Center manifold and stability for skew-product flows}, J. Dynam. Differential Equations, \textbf{6} (1994), 543-582.

\bibitem{CRo} R. Czaja and C. Rocha, Transversality in scalar reaction-diffusion equations on a circle, J. Differential Equations. \textbf{245} (2008), 692--721.

\bibitem{Fiedler1}
B. Fiedler, \rm{Discrete Ljapunov functionals and omega-limit sets}, RAIRO Modl. Math. Anal.
Numr. \textbf{23} (1989), 415-431.

\bibitem{Fiedler}
B.~Fiedler and J.~Mallet-Paret, \rm{A Poincar\'{e}-Bendixson theorem for
  scalar reaction diffusion equations}, Arch.
  Ration. Mech. Anal. \textbf{107} (1989), 325-345.

\bibitem{Fink}
A. Fink, \rm{Almost periodic differential equations}, Lecture Notes in Mathematics, Vol.~377, Springer 1974.

\bibitem{Hen}
 D. Henry, \rm{Geometric Theory of Semilinear Parabolic Equations}, Lecture Notes Mathematics Vol.840,
New York, Springer, 1981.

\bibitem{Hess}
P. Hess, \rm{Periodic-parabolic boundary value problems and positivity}, Pitman Research Notes in Mathematics Series, 247.

\bibitem{HuYi} W. Huang and Y. Yi, \rm{Almost periodically forced circle flows},
J. Funct. Anal. 257(2009), 832-902.

\bibitem{JR1} R. Joly and G. Raugel, \rm{Generic hyperbolicity of equilibria and periodic orbits of the parabolic equation on the circle}, Trans. Amer. Math. Soc. \textbf{362} (2010), 5189-5211.

\bibitem{JR2} R. Joly and G. Raugel, \rm{Generic Morse-Smale property for the parabolic equation on the circle}, Ann. Inst. H. Poincare, Anal. Non Lineaire \textbf{27} (2010), 1397-1440.


\bibitem{Massatt1986}
P.~Massatt, \rm{The Convergence of solutions of scalar reaction diffusion
  equations with convection to periodic solutions}, Preprint, 1986.

\bibitem{Matano78}
H.~Matano, \rm{Convergence of solutions of one-dimensional semilinear parabolic equations},
J. Math. Kyoto Univ. \textbf{18} (1978), no. 2, 221-227.

\bibitem{H.MATANO:1982}
H.~Matano, \rm{Nonincrease of the lap-number of a solution for a
  one-dimensional semi-linear parabolic equation}, J. Fac. Sci. Univ. Tokyo
  Sect.IA. \textbf{29} (1982), 401-441.

\bibitem{Matano}
H.~Matano, \rm{Asymptotic behavior of solutions of semilinear heat equations on
  $S^1$}. Nonlinear Diffusion Equations and Their Equilibrium States II, Math. Sci. Res. Inst. Publ. \textbf{13}  (1988), 139-162.

\bibitem{Mierczynski}
J. Mierczy\'nski and W. Shen, \rm{Spectral theory for random and nonautonomous parabolic equations and applications}, Chapman \& Hall/CRC Monographs and Surveys in Pure and Applied Mathematics, 139. CRC Press, 2008.


\bibitem{Sacker1978}
R. Sacker and G. Sell, \rm{A spectral theory for linear differential systems}, J.  Diff. Eqns. \textbf{27} (1978), 320-358.

\bibitem{Sacker1991}
R. Sacker and G. Sell, \rm{Dichotomies for linear evolutionary equations in Banach spaces}, J. Diff. Eqns. \textbf{113} (1991), 17-67.

\bibitem{SF1}
B.~Sandstede and B.~Fiedler, \rm{Dynamics of periodically forced parabolic
  equations on the circle}, Ergodic Theory and Dynamical Systems \textbf{12}
  (1992), 559--571.

\bibitem{Sell} G. Sell, Topological dynamics and ordinary differential equations,
 Van Nostrand Reinhold Co., Van Nostrand Reinhold Mathematical Studies, No.33. 1971.

\bibitem{SWZ2}
W. Shen, Y. Wang and D. Zhou, \rm{Structure of $\omega$-limit sets for almost-periodic parabolic equations on $S^1$ with reflection symmetry}, J.  Diff. Eqns. \textbf{267} (2016), 6633-6667.

\bibitem{SWZ3}
W. Shen, Y. Wang and D. Zhou, \rm{Long-time behavior of almost periodically forced parabolic equations on the circle}
J. Differential Equations \textbf{266} (2019), 1377-1413.

\bibitem{SWZ}
W. Shen, Y. Wang and D. Zhou, \rm{Almost automorphically and almost periodically forced circle flows of almost periodic parabolic equations on $S^1$}, J. Dyn. Diff. Eqns. 2019, 43p, https://doi.org/10.1007/s10884-019-09786-7.


\bibitem{Shen1995114}
W. Shen and Y. Yi, \rm{Dynamics of almost periodic scalar parabolic
  equations}, J.  Diff. Eqns. \textbf{122} (1995), 114-136.

\bibitem{ShenYi-2}
W. Shen and Y. Yi, \rm{Asymptotic almost periodicity of scalar parabolic equations with almost
periodic time dependence}, J.  Diff. Eqns. \textbf{122} (1995), 373-397.

\bibitem{ShenYi-TAMS95}
W. Shen and Y. Yi, \rm{On Minimal Sets of Scalar Parabolic Equations with Skew-product
Structures}, Tran. Amer. Math. Soc., \textbf{347} (1995), 4413-4431.

\bibitem{ShenYi-JDDE96}
W. Shen and Y. Yi, \rm{Ergodicity of minimal sets in scalar parabolic equations}, J. Dyn. Diff. Eqns.
\textbf{8} (1996), 299-323.

\bibitem{Shen1998}
W. Shen and Y. Yi, \rm{Almost automorphic and almost periodic dynamics in
  skew-product semiflows}, Mem. Amer. Math. Soc.
  \textbf{136} (1998), no.~647.


\bibitem{Te}
I. Tere\v{s}\v{c}\'{a}k, \rm{Dynamical systems with discrete Lyapunov functionals}, Ph.D. thesis,
Comenius University (1994).

\end{thebibliography}
\end{document}